\documentclass[11pt,a4paper]{article}

\usepackage{amsmath}
\usepackage[T1]{fontenc}
\usepackage[utf8]{inputenc}
\usepackage{amsfonts}
\usepackage{amsthm}
\usepackage{amstext}
\usepackage{amssymb}
\usepackage{color}
\usepackage{bbold}
\usepackage{hyperref}
\usepackage{url}

\newcommand{\OW}{\mathrm{OW}}
\newcommand{\OT}{\mathrm{OT}}
\newcommand{\SODECO}{\mathrm{SODECO}}
\newcommand{\ODECO}{\mathrm{ODECO}}

\newcommand{\cc}{\mathbb{C}}
\newcommand{\rr}{\mathbb{R}}

\newcommand{\diag}{\operatorname{diag}}
\newcommand{\rk}{\operatorname{rank}}

\newtheorem{fact}{Fact}
\newtheorem{theorem}{Theorem}
\newtheorem{example}[theorem]{Example}

\newtheorem{lemma}[theorem]{Lemma}
\newtheorem*{lemma*}{Lemma}
\newtheorem{proposition}[theorem]{Proposition}
\newtheorem{corollary}[theorem]{Corollary}
\newtheorem{remark}[theorem]{Remark}

\newtheorem*{open*}{Open~question}
\newtheorem{definition}[theorem]{Definition}

\begin{document}

\title{Orthogonal tensor decomposition and orbit closures
  from a linear algebraic perspective}

\author{Pascal Koiran\footnote{Univ Lyon, EnsL, UCBL, CNRS,  LIP, F-69342, LYON Cedex 07, France. Email: {\tt pascal.koiran@ens-lyon.fr}. This work got started at the Simons Institute for the Theory of Computing during the Fall 2018 program on Lower Bounds in Computational Complexity.}}

\maketitle

\begin{abstract}
  We study orthogonal decompositions of symmetric and ordinary
  tensors using methods from linear algebra. For the field of real numbers
  we show that the sets of decomposable tensors can be defined by equations of degree 2. This gives  a new proof of some of the results of Robeva and Boralevi et al.
  Orthogonal decompositions over the field of complex numbers had not
  been studied previously; we give an explicit description of the set
  of decomposable tensors using polynomial equalities and inequalities,
  and we begin a study of their closures.

  The main open problem that arises from this work is to obtain a
  complete description of the closures.
  This question is akin to that of characterizing
  {\em border rank} of tensors
  in algebraic complexity.
  We give partial results using in particular a connection with
  approximate simultaneous diagonalization
  (the so-called {\em ASD property}).

  \end{abstract}

\newpage

{\tableofcontents}

\newpage

\section{Introduction}

In this paper we study several types of {\em orthogonal tensor decompositions}
and give algebraic characterizations of the set of decomposable tensors.
That is, we give explicit systems of polynomial equations whose zero set
is the set of decomposable tensors. When that is not possible, i.e.,
when the set of decomposable tensors is not Zariski closed, we describe
it as a constructible set (using polynomial inequalities in addition
to polynomial equalities). In the non-closed case we begin a study
of the closure. The main goal here would be to obtain an explicit
description of the closure, and we give partial results in this direction.

The decompositions that we study can be defined in two equivalent
languages: the language of tensors, and the language of polynomials.
Indeed, as is well known one can associate to a symmetric tensor
(respectively, to an ordinary tensor) a homogeneous polynomial
(respectively, a multilinear polynomial) in the same way that
a quadratic form is associated to a symmetric matrix and a bilinear form
is associated to an arbitrary matrix.
We begin with a definition in the language of polynomials, but we will
switch between the two languages whenever that is convenient.

Let $K$ be a field of characteristic 0. We denote by $K[x_1,\ldots,x_n]_d$ the space of homogeneous polynomials of degree $d$ in $n$ variables (also called: ``degree $d$ forms''). The  cases of interest for this paper are $K=\rr$ and
$K=\cc$.
Recall that a matrix $A \in M_n(K)$ is said to be orthogonal if
$A^TA = \mathrm{Id}_n$.
\begin{definition} \label{waringdef}
We say that $f \in K[x_1,\ldots,x_n]_d$ admits an {\em orthogonal Waring decomposition}
if it can be written as $f(x)=g(Ax)$ where $A$ is an orthogonal matrix and $g$ is any polynomial
of the form 
\begin{equation} \label{powersum}
g(x_1,...,x_n)=\alpha_1x_1^d+\cdots+\alpha_nx_n^d
\end{equation}
with $\alpha_1,\ldots,\alpha_n \in K$.
\end{definition}
In this paper we focus on  
the case $d=3$ , which corresponds to symmetric tensors of order 3.
We will denote by $\OW_n(K)$ the set of homogeneous 
polynomials of degree $3$ in $n$ variables that admit such a decomposition,
and we will identify it with the corresponding set of symmetric tensors.

For $K=\rr$, Definition~\ref{waringdef} turns out to be
equivalent to the notion of
{\em symmetrically odeco} tensor studied in~\cite{boralevi17}.
According to~\cite{boralevi17}, a symmetric tensor is symmetrically odeco
if it can be written as
$$\sum_{i=1}^k 
\alpha_i v_i^{\otimes d}$$
where $\alpha_i = \pm 1$ and
$v_1,\ldots,v_k$ are nonzero, pairwise orthogonal vectors in~$\rr^n$. 

Let us now move to ordinary tensors. As recalled above, an ordinary tensor of order three
    $T \in K^{n \times n \times n}$ can be represented
by a trilinear form $t(x,y,z)=\sum_{i,j,k=1}^n T_{ijk}x_iy_jz_k$ where
$x,y,z$ denote three $n$-tuples of variables.
\begin{definition} \label{orthodef}
  We say that the trilinear form $t(x,y,z) \in K[x,y,z]$,
  or the corresponding tensor $T$,
admits an orthogonal decomposition if one can write $t(x,y,z)=g(Ax,By,Cz)$
where $A,B,C$ are orthogonal matrices and $g$ is a diagonal trilinear form in $3n$ variables,
i.e., a polynomial of the form:
\begin{equation} \label{diagtensor}
  g(x,y,z)=\sum_{i=1}^n \alpha_i x_i y_i z_i
\end{equation}
with $\alpha_1,\ldots,\alpha_n \in K$.
\end{definition}
We will denote by $\OT_n(K)$ the set of trilinear forms that admit an
orthogonal decomposition, and we will use the same notation for
the corresponding set of tensors.
For $K=\rr$, it turns out that Definition~\ref{orthodef} agrees with the definition
of an ``odeco tensor''
from~\cite{boralevi17}: an order 3 tensor is odeco
if it can be written as   $$\sum_{i=1}^k u_i \otimes v_i \otimes w_i$$
    where each of the the 3 lists $(u_1,\ldots,u_k)$, $(v_1,\ldots,v_k)$,
    $(w_1,\ldots,w_k)$ is made of $k$ nonzero, pairwise orthogonal vectors
    in $\rr^n$.

    \subsection{Results and methods}

    A tensor $T$ of order 3 and size $n$ can be viewed as $n$ matrices
    (the ``slices'' of $T$)  stacked on top of each other.
    This very down-to-earth
    point of view turns out to be remarkably powerful for the study
    of orthogonal decompositions because it allows us to leverage
    the known body of work on simultaneous reduction of matrices.
    For some of our results about the complex field we also need elements
    of the theory of quadratic forms (and in particular the notions
    of isotropic vectors, totally isotropic subspaces,
    Witt's extension theorem\ldots\footnote{As a side remark, isotropic
      subspaces also appear in the recent paper~\cite{bei19}
      where a connection is made
      to graph-theoretic concepts such as the chromatic number and the independence number.}).
    Our contributions are twofold:
    \begin{itemize}
    \item[(i)] We give more elementary proofs of two results
      from~\cite{boralevi17}: the set of odeco and symmetrically odeco tensors
      can be described by equations of degree 2. In~\cite{boralevi17}
      the result for symmetrically odeco tensors is obtained as follows:
      given a symmetric tensor $S$ of order 3, they define a bilinear map from
      $\rr^n \times \rr^n$ to $\rr^n$ associated to $S$ in a natural way;
      then they show that $S$ is symmetrically
      odeco iff this map 
      gives rise to an associative algebra.
      We obtain an alternative, intuitively appealing characterization:
      $S$ is symmetrically odeco iff its slices commute.
      We give two different proofs of this fact (see Section~\ref{realwaring}
      for details).
      
    \item[(ii)] We initiate a study of orthogonal decomposition over
      the field of complex numbers, which up to now had not been studied
      either from an algorithmic or structural point of view.
      In particular, we stress that all the results in~\cite{boralevi17}
      for the field of complex numbers are obtained for unitary rather than
      orthogonal decompositions.

      In this paper we give characterizations of the set of symmetric and
      ordinary tensors that admit orthogonal decompositions in the sense
      of Definitions~\ref{waringdef} and~\ref{orthodef}.
      In particular, for  orthogonal Waring decomposition over
$\cc$ the slices must commute and be diagonalizable.

      Compared to unitary or real orthogonal decompositions,
      a distinctive feature of complex orthogonal decompositions, and our
      main motivation for studying them, is that
      the set of decomposable tensors is not closed (as could be guessed
      from       the diagonalizability condition
      in the above characterization).
      We elaborate on this in Section~\ref{orbitclosure}.
    \end{itemize}

    A remark is in order regarding the above point (i).
    Jan Draisma (personal communication) has pointed out that the authors
    of~\cite{boralevi17} first worked in the language of slices before
    switching to the language of algebras, which turned out to be more
    efficient for their purposes. It may still be useful to present
    the first point of view in detail here since there does not remain any
    trace of it in~\cite{boralevi17}. More importantly, this
    point of view seems better suited to the study of complex orthogonal
    decompositions because, as explained in (ii),
    the sets of decomposable tensors are not closed.

    \subsection{Orbit closures} \label{orbitclosure}

    According to Definition~\ref{waringdef}, the set of decomposable
    tensors is the orbit of the set of polynomials of the form~(\ref{powersum})
    under the action of the orthogonal group. The closure of the
    set of decomposable tensors is therefore an {\em orbit closure}.
    We have a similar situation in Definition~\ref{orthodef} with the
    action of a product of 3 orthogonal groups.
    The notion of orbit closure plays a central in Mulmuley and Sohoni's
    Geometric Complexity Theory~\cite{mulmuley01,mulmuley08}.
    In their program, the goal is to show that the permanent polynomial\footnote{More precisely, a ``padded'' version of the permanent.}
    is not in the orbit closure of the determinant (under the action
    of the general linear group).
    
    Closer to the topic of the present paper, we find the notion of
    {\em border rank} of tensors which plays an important role in the
    the study of matrix multiplication algorithms~\cite{BCS}.
   If we replace in Definition~\ref{orthodef} the orthogonal group
    by the general linear group, the corresponding orbit closure
    is the set of tensors of border rank at most $n$.
    The same change in Definition~\ref{waringdef} would yield the set
    of polynomials of {\em border Waring rank} at most $n$.
    Restricting to the orthogonal group as we do in this paper yields
    orbit closures which should hopefully be easier to study.
    We study them ``from below'', i.e., we find nontrivial families of
    tensors which belong to the orbit closures;
    and we study them ``from above'', i.e., we find equations that must
    be satisfied by all tensors in the orbit closures. For instance,
    we show that the slices of any tensor in the closure  $\overline{\OW_n(\cc)}$ of $\OW_n(\cc)$
    must commute, that they generate a matrix algebra of dimension at
    most $n$, and that their centralizer is of dimension at least $n$.
    We obtain these equations thanks to a connection with the so-called
    ``ASD property''~\cite{omeara06,omeara11}:
    \begin{definition} \label{asd}
  A tuple $(A_1,\ldots,A_k)$ of matrices in $M_n(\cc)$
  is {\em approximately simultaneously diagonalizable} (ASD for short)
  if it is in the closure of the set of $k$-tuples
  of simultaneously diagonalizable matrices, i.e., if for any $\epsilon>0$
  there exist  simultaneously diagonalizable matrices $B_1,\ldots,B_k$
  which satisfy $||A_i - B_i|| < \epsilon$ for $i=1,\ldots,k$.
    \end{definition}
    Indeed, we show that the slices of a tensor in $\overline{\OW_n(\cc)}$
    must satisfy the ASD property.

\subsection{Algorithmic issues}

The characterizations that we obtain in our 4 main scenarios
(orthogonal Waring decomposition of symmetric tensors and orthogonal decomposition of ordinary tensors, over the real and complex fields)
are straightfoward to check algorithmically:
they involve only standard linear-algebraic computations on the slices
of the tensors (as a previously mentioned example, for orthogonal Waring decomposition
over
$\cc$ the slices must commute and be diagonalizable).
From these  characterizations one could also
derive algorithms that effectively
construct a decomposition whenever that is possible.
These algorithms would rely on standard routines from linear algebra
such as simultaneous matrix diagonalization.
We will not go into the details 
in the present paper
but we note that there is a large literature on algorithms for various
types of tensor decompositions
(see e.g.~\cite{batselier15,kolda15,salmi09,zhang01}),
and that they are often based on linear algebra.
In particular, the preprint~\cite{kolda15} makes the connection between
simultaneous diagonalization and orthogonal decomposition of real symmetric
tensors.
Therefore, one contribution of our paper is to show that techniques
from linear algebra are not only useful for designing decomposition
algorithms,
but also to obtain algebraic characterizations.\footnote{We note 
  that the connection between singular value decomposition
    and tensor decomposition
algorithms is made in~\cite{boralevi17} (see Proposition~7 from that paper, and the remarks
thereafter). However, the algebraic characterizations obtained later in~\cite{boralevi17} 
do not rely on the SVD or similar techniques.}

\subsection{Open problems}

The main open problem that arises from this work is the complete
determination of the orbit closures $\overline{\OW_n(\cc)}$
and $\overline{\OT_n(\cc)}$. Indeed, as shown later in the paper
the study of the orbit closures ``from below'' and ``from above'' do not lead
to a complete characterization.
This question bears a certain similarity to another open problem
at the intersection of linear algebra and algebraic geometry:
obtaining a complete characterization of the ASD
property~\cite{omeara06,omeara11}.
For application to orthogonal tensor decompositions, we note that
it suffices to study the ASD property for tuples of symmetric
matrices, and furthermore to consider only approximations by 
tuples of simultaneously diagonalizable {\em symmetric} matrices.
This holds true even for ordinary tensors, see Proposition~\ref{otasd}
and Remark~\ref{rem:otasd} at the end of the paper.

As explained in Section~\ref{orbitclosure},
if the orthogonality requirement is lifted from our tensor decompositions
one obtains instead of $\overline{\OW_n(\cc)}$ and $\overline{\OT_n(\cc)}$
the sets of tensors of border (Waring) rank at most $n$.
Obtaining explicit equations for these sets is vey much an open problem,
with applications to lower bounds in algebraic complexity theory
(see e.g.~\cite{landsbergGCT}).

In this paper we have studied decompositions of order 3 tensors.
In~\cite{boralevi17}, tensors of higher order are handled by reduction
to the case of order 3. Namely, they show that a tensor of order $d \geq 4$
admits a unitary or real orthogonal decomposition iff the same is true
for certain flattenings of the tensor. It would be interesting to find out
whether a similar property holds for complex orthogonal decompositions.
Following this approach would also require a generalization of our results
to ``rectangular'' tensors of order 3. Indeed, we have only studied
``cubic'' tensors (of format $n \times n \times n$).
But even if we start from a higher order cubic tensor (e.g., a tensor $T$
of order 4 and format $n \times n \times n \times n$), its flattenings will not
be cubic in general. For instance, we would obtain from $T$ flattenings
of format $n \times n \times n^2$.

\section{Background}

In this section we first present some background on tensors and matrices.
Indeed, as explained in the introduction simultaneous reduction of matrices
(and in particular simulataneous diagonalization) plays an important
role in this paper.
For the study of complex tensors we will also need some elements of
the theory of quadratic forms, which we present
in Section~\ref{quadratic}.

\subsection{Tensors and their slices} \label{tensorbackground}

An order 3 tensor $T \in K^{n \times n \times n}$
can be represented by the trilinear form
\begin{equation} \label{trilinear}
  t(x,y,z)=\sum_{i,j,k=1}^n T_{ijk}x_iy_jz_k
\end{equation}
where $x,y,z$ denote three $n$-tuples of variables.
There are 3 ways of decomposing $T$ into a tuple of $n$ matrices:
we can decompose in the $x$, $y$ or $z$ direction. We call the resulting
matrices the $x$-slices, $y$-slices, and $z$-slices.
For instance, the $z$-slices  are
the matrices of the bilinear forms $\partial t / \partial z_k$
($1 \leq k \leq n$).

The tensor $T$ is said to be symmetric if it is invariant under all of the six
permutations of the indices $i$, $j$, $k$. For such a tensor the $x$, $y$ and $z$ slices are identical and we simply call them ``the $n$ slices of $T$'';
each of these slices is a symmetric matrix.
To a symmetric tensor $T$ we associate the degre 3 form
$f(x_1,\ldots,x_n) = \sum_{i,j,k=1}^n T_{ijk} x_i x_j x_k$.
Note that $f$ is obtained from the trilinear form~$t$  in~(\ref{trilinear})
by setting $x=y=z$, in the same way that a quadratic form is obtained from
a bilinear form.
Since we switch freely between the language of tensors and the language of polynomials (for slices and for other notions), by ``slices of $f$'' we will mean
the slices of the corresponding symmetric tensor $T$.

In light of Definitions~\ref{waringdef} and~\ref{orthodef}, it is important
to understand how slices are affected by a linear change of variables.
Let us do this for symmetric tensors: 
we will show that the slices $S_1,\ldots,S_n$ of a symmetric tensor
$S \in \OW_n(K)$ are given by:
\begin{equation} \label{simslice}
  S_k = A^TD_kA,\ D_k=\diag(\alpha_1 a_{1k},\ldots,\alpha_n a_{nk}).
\end{equation}
Our proof  will use the following property of Hessian matrices~(\cite{Kayal11}, Lemma 5.1).
\begin{fact} \label{hessianfact}
Let  $G$ be an $n$-variate polynomial and $A \in K^{n \times n}$ a linear transformation.
 Let $F(x)=G(Ax)$. The Hessian matrices of $F$ and $G$ satisfy the relation:
 \begin{equation}  \label{hessian}
 H_F(x) = A^TH_G(Ax)A.
 \end{equation}
\end{fact}
Consider any homogeneous polynomial of degree 3. The entries of $H_f$ are linear forms 
in $x_1,\ldots,x_n$. One can therefore write $H_f=x_1A_1+\cdots+x_nA_n$ where the $A_i$ 
are $n \times n$ matrices with entries in $K$. These matrices are symmetric since 
entry $(j,k)$ of $A_i$ is given by: 
\begin{equation}\label{partial3}
\displaystyle (A_i)_{jk} = \frac{\partial^3 f}{\partial x_i \partial x_j \partial x_k}.
\end{equation}
Each entry of $A_i$ therefore corresponds to a monomial of $f$. More precisely, $(A_i)_{jk}$
is obtained by multiplication of the coefficient of $x_i x_j x_k$ by 1, 2 or 6.
Equivalently, one can define $A_i$ as the Hessian matrix of $\partial f / \partial x_i$.
\begin{remark}
  It follows from~(\ref{partial3}) that the  $A_i$ are scalar multiples
  of the slices of the symmetric tensor associated to $f$,
  and the constant of proportionality is equal to 6.

  We illustrate this point on the example of the polynomial
  $f(x_1,x_2)=(x_1-x_2)^3$. The corresponding tensor is the rank 1 symmetric
  tensor $T=e^{\otimes 3}$ where $e=(1,-1)$.
  The entries of $T$ are: $T_{111}=1$, $T_{112}=-1$, $T_{122}=1$ and $T_{222}=-1$
  (the other entries are obtained from these 4 entries by symmetry of $T$).
  The two slices of $T$ are the matrices:
  $$T_1=
  \begin{pmatrix}
  1 & -1\\
  -1 & 1
  \end{pmatrix},
  T_2=-T_1
  $$
  and it is easy to check that the Hessian matrices of $\partial f / \partial x_1, \partial f / \partial x_2$ are respectively $6T_1$ and $6T_2$.
\end{remark}
\begin{lemma} \label{diaghessian}
The Hessian matrix of a polynomial $g \in K[x_1,\ldots,x_n]_3$ is diagonal if and only if $g$
is of the form (\ref{powersum}). 
\end{lemma}
\begin{proof}
If $g$ is of the required form, $H_g=\diag(6\alpha_1x_1,\ldots,6\alpha_nx_n)$. Conversely, if $H_g$
is diagonal, $g$ can only contain cubes of powers since any other monomial would give rise
to an off-diagonal term in $H_g$.
\end{proof}
We can now establish formula~(\ref{simslice}):
\begin{proposition}
  The slices $S_1,\ldots,S_n$
  of a polynomial $f\in K[x_1,\ldots,x_n]_3$ as in Definition~\ref{waringdef}
  are given by~(\ref{simslice}).
\end{proposition}
\begin{proof}
  The Hessian matrix of the polynomial $g=\alpha_1x_1^3+\ldots+\alpha_n x_n^3$ is
  $H_g=\diag(6\alpha_1x_1,\ldots,6\alpha_nx_n)$.
  Hence the result follows from Fact~\ref{hessianfact} since $A_k = 6S_k$.
\end{proof}
We will establish a similar result for ordinary tensors
in Proposition~\ref{slices}.
In particular, we will show
(following the notations of Definition~\ref{orthodef}) that the $z$-slices $T_1,\ldots,T_n$  of a trilinear form $t \in \OT_n(K)$ are given by the formula
\begin{equation} \label{sliceforward}
 T_k = A^TD_kB,\ D_k=\diag(\alpha_1 c_{1k},\ldots,\alpha_n c_{nk})
\end{equation}
where the $c_{ik}$ are the entries of $C$.
Note that~(\ref{simslice}) can also be derived from  this formula.
Consider indeed a degree 3 form $f \in \OW_n(K)$.
By Definition~\ref{waringdef}, $f(x)=h(x,x,x)$ where $h$ is the trilinear form 
$g(Ax,Ay,Az)$ and $g$ is as in~(\ref{orthodef}).
Viewed as an ordinary tensor, a symmetric tensor in $\OW_n(K)$ therefore
belongs to $\OT_n(K)$\footnote{As a sanity check, one can easily verify that
  if a symmetric tensor satisfies the conditions for membership in
  $\OW_n(\rr)$ (Theorem~\ref{realth}) or $\OW_n(\cc)$ (Theorem~\ref{complexth}) then it also satisfies the conditions for membership in
$\OT_n(\rr)$ (Theorem~\ref{realortho}) or $\OT_n(\cc)$ (Theorem~\ref{cortho}).}
and we can obtain~(\ref{simslice})
by setting $A=B=C$ in~(\ref{sliceforward}).

Finally, we mention a consequence of Fact~\ref{hessianfact} (Lemma~5.2 in~\cite{Kayal11}) which will be
useful in Section~\ref{sclosure}.
\begin{lemma} \label{lem:hessian}
  Let $f \in K[x_1,\ldots,x_n]$ be a polynomial of the form
  $$f(x_1,\ldots,x_n)=\sum_{i=1}^n a_i\ell_i(x_1,\ldots,x_n)^d$$
  where the $\ell_i$ are linearly independent linear forms, the $a_i$ are
  nonzero constants and $d \geq 2$. Then the Hessian determinant of $f$
  satisfies
  $$\det H_f(x_1,\ldots,x_n) = c \prod_{i=1}^n \ell_i(x_1,\ldots,x_n)^{d-2}$$
  where $c \in K$ is a nonzero constant. In particular, $\det H_f$ is not
  identically~0.
\end{lemma}

\subsection{Diagonalization and constructible sets} \label{constructible}

One of the main goals of the paper is to give necessary and sufficient
conditions for membership of symmetric tensors in $\OW_n(\cc)$ and
of ordinary tensors in $\OT_n(\cc)$.
In particular, certain matrices (slices, or products of slices)
must be diagonalizable.
Since the set of diagonalizable matrices is dense in $M_n(\cc)$,
it cannot be written as the zero set of of a system of polynomial equations.
It can however be described by polynomial equalities and inequalities, i.e.,
it is a {\em constructible subset} of $M_n(\cc)$.
This follows from the following well known result:
\begin{proposition} \label{diagmat}
  Let $K$ be a field of characteristic 0 and let 
  $\chi_M$ be the characteristic polynomial of a matrix $M \in M_n(K)$.
  Let $P_M = \chi_M / \mathrm{gcd}(\chi_M,\chi_M')$ be the squarefree part of
  $\chi_M$. The matrix $M$ is diagonalizable over $\overline{K}$ iff
  $P_M(M)=0$.
\end{proposition}
Proposition~\ref{diagmat} together with the characterizations in Sections~\ref{csection} and~\ref{sec:cordi} show that the sets
of tensors $\OW_n(\cc)$ and $\OT_n(\cc)$ are constructible.

If we restrict to symmetric matrices, it is still true that 
diagonalizable matrices are dense.
This fact will be used in Section~\ref{sclosure}.
For the sake of completeness we give the (standard) proof below.
\begin{lemma} \label{diagsymdense}
  The set of diagonalizable symmetric matrices is dense in the
  set of complex symmetric matrices.
\end{lemma}
\begin{proof}
  We will prove a stronger result: the set of symmetric matrices with $n$
  distinct eigenvalues is dense in the set of symmetric matrices of size $n$.
  In order to show this we associate to a symmetric matrix $S$
  of size $n$ the discriminant $\mathrm{Disc}_S$ of its characteristic
  polynomial. Our matrix has $n$ distinct eigenvalues if and only if
  $\mathrm{Disc}_S \neq 0$. Note that this discriminant can be viewed
  as a polynomial in the $n(n+1)/2$ upper triangular entries of $S$.
  Therefore, the conclusion will follow if we can show that this polynomial
  is not identically 0. This is clear since $\mathrm{Disc}_S \neq 0$
  if we take $S$ to be (for instance) a diagonal matrix with distinct
  diagonal entries.
 \end{proof}

\subsection{Simultaneous diagonalization} \label{sec:simdiag}

As mentioned in the introduction several of our results hinge
on simultaneous reduction of matrices, and in particular
on simultaneous diagonalization.
\begin{lemma} \label{lincomb}
  Let $A_1,\ldots,A_k \in M_n(K)$ be a tuple of simultaneously diagonalizable
  matrices, and let $S \subseteq K$ be a finite set of size $|S| > n(n-1)/2$.
  Then there exist $\alpha_2,\ldots,\alpha_k$ in $S$ such that any transition
  matrix which diagonalizes $A_1+\alpha_2 A_2 + \ldots + \alpha_k A_k$
  must also diagonalize all of the matrices $A_1,\ldots, A_k$.
\end{lemma}
\begin{proof}
  We proceed by induction on $k$.
  The base case $k=2$ is Proposition~2 in~\cite{koiran2018orbits}.
  Assume now that the result holds true at step $k-1$.
  By induction hypothesis there exist $\alpha_2,\ldots,\alpha_{k-1}$ in $S$ such that any transition
  matrix which diagonalizes $M=A_1+\alpha_2 A_2 + \ldots + \alpha_{k-1} A_{k-1}$
  must also diagonalize $A_1,\ldots, A_{k-1}$.
  Moreover, by the base case there exists $\alpha_k \in S$ such that
  any transition matrix $T$ which diagonalizes $M+\alpha_k A_k$
  must diagonalize $M$ and~$A_k$. Therefore such a $T$ must diagonalize
  all of the $A_1,\ldots,A_k$.
\end{proof}
This lemma has the following important consequence.
\begin{theorem} \label{th:simdiag}
  Let $A_1,\ldots,A_k$ be a tuple of symmetric matrices of  $M_n(K)$ where $K=\rr$ or $K=\cc$.
  If the $A_i$ are simultaneously diagonalizable then they are  simultaneously diagonalizable by an orthogonal change of basis, i.e., there is a real (respectively, complex)  orthogonal matrix $P$ such that the $k$ matrices $P^TA_iP$ are diagonal.
 \end{theorem}
\begin{proof}
  We begin with $K=\rr$. Let us fix $\alpha_2,\ldots,\alpha_k$ as in
  the previous lemma. Since the matrix
  $S=A_1+\alpha_2 A_2 + \ldots + \alpha_k A_k$
  is real symmetric, it can be diagonalized by a real orthogonal matrix.
  By the lemma, such a matrix will diagonalize all of the $A_i$.

  For $K=\cc$, the difference with the real case is that it is no longer true
  that all symmetric matrices are diagonalizable
  (see e.g. Example~\ref{borderex} and the beginning of Section~\ref{csection}\footnote{The existence of such examples is due to the presence of isotropic vectors in $\cc^n$; more on this topic in Section~\ref{quadratic}.}).
  It is however still the case that if a complex symmetric matrix is
  diagonalizable, then it can be diagonalized by an orthogonal matrix
  (Theorem 4.4.27 of~\cite{horn13}).
  In the present situation, $S$ will be diagonalizable for any choice
  of complex numbers $\alpha_2,\ldots,\alpha_k$ since the $A_i$ are 
  simultaneously diagonalizable.
  The result therefore follows from Lemma~\ref{lincomb} like in the real case.
\end{proof}
It is possible to give a direct (non inductive) proof of this theorem based on Theorem 6.4.16, Corollary 6.4.18 and Corollary 6.4.19 of \cite{topics}. Moreover, this argument shows that the conclusion of Theorem \ref{th:simdiag} holds not only when the $A_i$ are symmetric, but also when they are all skew-symmetric or all orthogonal (Roger Horn, personal communication).

\subsection{Quadratic forms} \label{quadratic}

As mentioned before, for $K=\cc$ we need some elements of the theory
of quadratic forms. We will work only
with the quadratic form $\sum_{i=1}^n x_i^2$
on $\cc^n$ and the associated bilinear map 
$\langle x,y \rangle = \sum_{i=1}^n x_iy_i$
but much of what follows applies to an arbitrary
nondegenerate quadratic space.
We will refer to this bilinear map as the ``Euclidean inner product'',
but this is an 
abuse of terminology since we are working over the field of complex numbers.
Compared to the case $K=\rr$ the main complication (and the reason why
one often works with Hermitian and unitary matrices rather than
symmetric and orthogonal matrices) is that $\cc^n$ contains isotropic vectors.
\begin{definition} \label{def:isotropic}
  A vector $v \in \cc^n \setminus \{0\}$ is isotropic if it is self-orthogonal,
  i.e., if $\langle v,v \rangle =0$.

  More generally,
  a subspace $V \subseteq \cc^n$, $V \neq \{0\}$
  is said to be {\em totally isotropic}
  if the Euclidean inner product is identically 0 on $V$, or equivalently
  if $V$ contains only isotropic vectors.
\end{definition}
Whether the null vector is defined to be isotropic or not is a matter of convention.
\begin{theorem} \label{supplement}
  Let $U$ be a totally isotropic subspace of $\cc^n$ with basis
  $(u_1,\ldots,u_r)$. There exists another totally isotropic subspace $U'$,
  disjoint from $U$, with basis $(u'_1,\ldots,u'_r)$ such that
  $\langle u_i,u'_j \rangle = \delta_{ij}$.
 \end{theorem}
This lemma applies not only to $\cc^n$ but to any nondegenerate quadratic
space. For a proof see Theorem 6.2 in the lecture notes~\cite{clark},
where $U'$ is called
an ``isotropic supplement'' to $U$.
\begin{corollary} \label{total}
  Let $(u_1,\ldots,u_k)$ be a tuple of linearly independent pairwise
  orthogonal vectors of $\cc^n$.
  If there are $r$ isotropic vectors in this tuple
  then $r+k \leq n$. In particular, if $U$ is a totally isotropic subspace
  of $\cc^n$ then $\dim U \leq n/2$.
\end{corollary}
\begin{proof}
  Assume for instance that $u_1,\ldots,u_r$ are the isotropic vectors
  in this tuple.
  Let $U$ be the (totally isotropic) subspace spanned by $u_1,\ldots,u_r$
  and let $V$ be the subspace spanned by the non-isotropic vectors
  $u_{r+1},\ldots,u_k$. Let $U'$ be the ``isotropic supplement'' to $U$
  provided by Theorem~\ref{supplement}.
  We claim that $U'$ is disjoint
  from $U \oplus V$. Suppose indeed that $u'=u+v$ with $u' \in U'$,
  $u \in U$, $v \in V$.
  Let us write $u'$ as $u'=\sum_{j=1}^r \alpha_j u'_j$ where
  $(u'_1,\ldots,u'_r)$ is the basis of $U'$ provided
  by Theorem~\ref{supplement}. For any $i \leq r$
  we have $\langle u',u_i \rangle = \alpha_i$
  since $\langle u_i,u'_j \rangle = \delta_{ij}$.
  On the other hand, we have
  $$\langle u',u_i \rangle = \langle u + v, u_i \rangle =
  \langle u , u_i \rangle + \langle v , u_i \rangle =0.$$
  In the last equality we have used the fact that $\langle u , u_i \rangle =0$
  (since $U$ is totally isotropic) and $\langle v , u_i \rangle =0$
  (since $u_1,\ldots,u_k$ are pairwise orthogonal).
  This proves the claim
  since we have shown that $\alpha_i=0$ for all $i$, i.e., $u'=0$.

  The conclusion of Corollary~\ref{total} follows directly from the claim:
  we have $U' \oplus U \oplus V \subseteq \cc^n$
  so $\dim U + \dim V + \dim U' \leq n$; but $\dim U = \dim U'=r$
  and $\dim V = k-r$.
\end{proof}
We will also use the following version of Witt's extension theorem,
a cornerstone of the theory of quadratic forms.
\begin{theorem} \label{witt}
  Let $U$ be a linear subspace of $\cc^n$. Any isometric embedding
  $\phi:U \rightarrow \cc^n$ extends to an isometry $F$ of $\cc^n$.
\end{theorem}
Here, ``isometric embedding'' means that $\phi$ is an injective linear
map which preserves the Euclidean inner product. Likewise, an
isometry of $\cc^n$ is a linear automorphism which respects the inner product
(i.e., $F$ is represented in the standard basis by an orthogonal matrix).
Witt's theorem applies not only to $\cc^n$ but to any nondegenerate quadratic
space. For a proof see e.g.
\cite[Theorem~42:17]{omeara} or \cite[Theorem~5.3 of Chapter~1]{scharlau}
or Corollary~7.4 in the lecture notes~\cite{clark}.
\begin{corollary} \label{wittcor}
  Let $v_1,\ldots,v_k$ be an orthornormal family of vectors of $\cc^n$ (i.e.,
  $\langle v_i,v_j \rangle = \delta_{ij}$). This family can be extended to an
  orthonormal basis $v_1,\ldots,v_n$ of $\cc^n$.
\end{corollary}
\begin{proof}
  Let $(e_1,\ldots,e_n)$ be the standard basis of $\cc^n$ and let $U$ be
  the subspace spanned by $e_1,\ldots,e_k$. The family $v_1,\ldots,v_k$
  is linearly independent since it is orthonormal. As a result, there is a
  (unique) isometric embedding $\phi:U \rightarrow \cc^n$ such that
  $\phi(e_i)=u_i$ for $i=1,\ldots,k$. Let $F$ be the isometry of $\cc^n$
  provided by Theorem~\ref{witt}. The desired orthonormal basis is
  $v_1=F(e_1),\ldots,v_n=F(e_n)$.
\end{proof}
Here is another useful consequence of Witt's theorem:
\begin{corollary} \label{orthosupplement}
  Let $(u_1,\ldots,u_k)$ be a tuple of linearly independent pairwise
  orthogonal vectors of $\cc^n$.
  Assume that $(u_1,\ldots,u_r)$ are the isotropic vectors in this list,
  and denote by $U$ the subspace that they span.
  
The ``isotropic supplement'' to $U$ provided by Theorem~\ref{supplement}
can be chosen to be orthogonal to the subspace $V$
spanned by $(u_{r+1},\ldots,u_k)$.
\end{corollary}
\begin{proof}
  Let $(e_1,\ldots,e_n)$ be the standard basis of $\cc^n$ and
  consider the vectors $w_1,\ldots,w_k$ defined as follows:
  $w_j=e_{2j-1}+ie_{2j}$ for $j \leq r$ and $w_j=e_{r+j}$
  for $r+1 \leq j \leq k$.
  These vectors are well defined since $k+r \leq n$
  according to Corollary~\ref{total}.
  We denote by $W$ the subspace that they span.
  
  The conclusion of Corollary~\ref{orthosupplement} is clear in the special case
  where $(u_1,\ldots,u_k)=(w_1,\ldots,w_k)$.
    Indeed, we can take $U'$  to be the space spanned
  by the vectors $w'_j=(e_{2j-1}-ie_{2j})/2$ ($1 \leq j \leq r$).
  We will reduce the general case to this one thanks to Witt's
  extension theorem.

  Consider then a linear map $\phi:W \rightarrow U \oplus V$ such that $\phi(w_j)=c_ju_j$ for $j \leq k$. Choose nonzero constants $c_j$ so that
  $\langle \phi(w_j),\phi(w_j) \rangle = \langle w_j, w_j \rangle$
  for all $j$ (we may, and will, take $c_j=1$ for $j \leq r$).
  This map is designed to preserve the inner product on $W$, and it is
  injective since the $u_j$ are linearly independent.
  It can thefore be extended to an isometry $F$ of $\cc^n$
  by Theorem~\ref{witt}. Then we take $U'$  to be the space spanned
  by the vectors $u'_j=F(w'_j)$.
\end{proof}

\section{Orthogonal Waring decomposition} \label{sec:waring}

Orthogonal Waring decomposition has   been studied in particular
in~\cite{robeva16,boralevi17}
where it is called {\em orthogonal decomposition of symmetric tensors}.
Theorem~\ref{realth} below provides an alternative and elementary treatment
for the case
of order 3 tensors.
Orthogonal decompositions are defined in~\cite{robeva16,boralevi17}
in the language of tensors rather than in the 
language of polynomials used in Definition~\ref{waringdef}.
The two definitions are indeed equivalent:
\begin{proposition} \label{symdefequiv}
  Let $f \in \rr[x_1,\ldots,x_n]_d$ be a homogeneous polynomial of degree $d$
  and let $S$ be the corresponding symmetric tensor of order $d$.
  The two following properties are equivalent:
  \begin{itemize}
  \item[(i)] $f$ admits an orthogonal Waring decomposition.
  \item[(ii)] $S$ is ``symmetrically odeco''~\cite{boralevi17},
    i.e., can be written as
    $$\sum_{i=1}^k \pm v_i^{\otimes d}$$
    where $v_1,\ldots,v_k$ are nonzero, pairwise orthogonal vectors
    in $\rr^n$.
  \end{itemize}
\end{proposition}
\begin{proof}
  Suppose that $f(x)=g(Ax)$ where $A$ is an orthogonal matrix and $g$ is
  as in~(\ref{powersum}).
  By definition,
  $f(x)=\sum_{i=1}^n a_i \langle v_i, x \rangle^d$
  where the $v_i$ are the rows of $A$.
  Therefore we have $S=\sum_{i=1}^n a_i v_i^{\otimes d}.$
  We obtain a decomposition of the form (ii) by dropping the terms
  with $a_i=0$, and dividing the remaining $v_i$ by $|a_i|^{1/d}$.

  Conversely, if $S$ is symmetrically odeco we can extend $v_1,\ldots,v_k$
  to an orthogonal basis $v_1,\ldots,v_n$ of $\rr^n$ and we can normalize
  these vectors to obtain an orthonormal basis $w_1,\ldots,w_n$.
  This yields a decomposition of the form
  $S=\sum_{i=1}^n a_i w_i^{\otimes d},$ and $f$ admits the orthogonal Waring decomposition $f(x)=\sum_{i=1}^n a_i \langle w_i, x \rangle^d.$
\end{proof}
Note that for odd $d$, the $\pm$ signs can be dropped  from (ii).

The above equivalence is
very straightforward but we point out that it fails over the field of complex
numbers, namely, it is no longer the case that (ii) implies (i).
This is due to the fact that some vectors $v_i \in \cc^n$ could be isotropic
(in the sense of Definition~\ref{def:isotropic})
and such vectors cannot be normalized:
consider for instance the polynomial $f=(x_1+ix_2)^3$ of Example~\ref{borderex}
and the corresponding tensor $S=(1,i)^{\otimes 3}$.
As a result, over $\cc$ we no longer have a single notion of
``symmetric orthogonal  decomposition.''
In Section~\ref{csection} we propose a natural version of (ii) for the
field of complex numbers, which we denote $\SODECO_n(\cc)$.
\footnote{The notation $\SODECO$ stands for ``symetrically odeco.''}
We investigate the relationship of this class of tensors with $\OW_n(\cc)$
in Section~\ref{csection} and with $\overline{\OW_n(\cc)}$
in Section~\ref{sclosure}. As one of our main results we will show that:
\begin{theorem} \label{firstinclosure}
For every $n \geq 1$ we have $$\OW_n(\cc) \subseteq \SODECO_n(\cc) \subseteq \overline{\OW_n(\cc)}.$$
  These two inclusions are strict for every $n \geq 2$.
\end{theorem}
As we will see in Section~\ref{ordi}, the situation for ordinary tensors
is similar: we have a single notion of ``orthogonal tensor decomposition''
over $\rr$ but not over $\cc$.

\subsection{Orthogonal Waring decomposition over the reals}
\label{realwaring}

\begin{theorem} \label{realth}
  A real symmetric tensor of order 3
  admits an orthogonal Waring decomposition if and only if its slices pairwise
  commute.
    In particular, the set of symmetric tensors of order 3 and size $n$
    that admit an orthogonal Waring
decomposition is the zero set of a system of $n^2(n-1)^2/4$ polynomial equations of degree 2 in $n+2 \choose 3$ variables.
\end{theorem}
\begin{proof}
  Let $f \in \rr[x_1,\ldots,x_n]_3$ be the degree 3 form associated to
  the symmetric tensor, and let $S_1,\ldots,S_n$ be the slices.
  We first consider the case where $f$  admits the orthogonal Waring decomposition $f(x)=g(Ax)$.
  Recall from~(\ref{simslice}) that the slices satisfy $S_k = A^TD_kA$
  where the matrices $D_k$ are diagonal. Since $A^T=A^{-1}$ the slices are
  simultaneously diagonalizable and they must therefore commute.

  For the converse, assume now that the slices commute.
  Recall that the matrices $A_1,\ldots,A_n$ in~(\ref{partial3})
  satisfy $A_k = 6S_k$.
  It is a well know fact of linear algebra
that a set of matrices is simultaneously diagonalizable iff these matrices commute  and each
matrix is diagonalizable (\cite{horn13}, Theorem 1.3.21). The latter assumption
is satisfied since the $A_i$ are real symmetric matrices.
By Theorem~\ref{th:simdiag} there are diagonal matrices $D_1,\ldots,D_n$  and an orthogonal matrix $A$ 
such that $A_i= A^TD_iA$. Therefore $H_f(x)=x_1A_1+\cdots+x_nA_n =  A^T D(x) A$ where 
$D(x)$ is a diagonal matrix whose entries are linear forms in $x_1,\ldots,x_n$.
Consider now the polynomial $g(x)=f(A^{-1}x)=f(A^Tx)$. 
By Fact \ref{hessian} we have 
$$H_g(x)=AH_f(Ax)A^T=D(Ax)$$
and $H_g$ is a diagonal matrix.
By Lemma \ref{diaghessian}, $g$ is of form (\ref{powersum}) and $f(x)=g(Ax)$ admits an orthogonal Waring decomposition.

The resulting polynomial system contains $n^2(n-1)^2/4$ equations because we have to express the commutativity of $n(n-1)/2$ pairs of matrices.  Each 
commutativity condition yields $n(n-1)/2$ equations
(Indeed, two symmetric matrices commute iff their product
is symmetric as well; it therefore suffices to express the equality of each
upper triangular entry of the product with the corresponding lower triangular
entry).
\end{proof}
In the remainder of Section~\ref{realwaring} we elucidate the connection
between the linear algebraic approach leading to Theorem~\ref{realth}
and the approach from~\cite{boralevi17}.
Given a 3-dimensional symmetric tensor $T$,
the authors of that paper define on $V=\rr^n$ a certain bilinear map
$V \times V \rightarrow V$, $(u,v) \mapsto u.v$.
This map is defined on elements of the standard basis $e_1,\ldots,e_n$ of $\rr^n$
by:
\begin{equation} \label{law}
  e_i.e_j = \sum_{l=1}^n T_{ijl}e_l.
  \end{equation}
One can then extend this map to the whole of $V \times V$ by bilinearity
(in~\cite{boralevi17} they actually
give an equivalent coordinate free definition of this map).
Then they establish a connection between the associativity of
this map and symmetric odeco decompostions:
\begin{theorem} \label{boralevi-sodeco}
The tensor $T$ is symmetrically odeco if and only if $(V,.)$ is associative.
\end{theorem}
In light of Theorem~\ref{realth} and Proposition~\ref{waringtensor},
we can conclude from Theorem~\ref{boralevi-sodeco}
that $(V,.)$ is associative if and only if the slices of $T$ commute.
We now give an alternative  proof of this equivalence by a direct calculation.
In fact, we show that this equivalence holds for an arbitrary field
(note indeed that~(\ref{law}) makes sense for an arbitrary field and not just
for the field of real numbers).
\begin{theorem}
  Let $T$ be a symmetric tensor of order 3 and size $n$, with entries
  in an arbitrary field $K$. The slices of $T$ commute if and only if
  $(K^n,.)$ is associative.
\end{theorem}
\begin{proof}
  Assume first that $(K^n,.)$ is associative. In particular, for vectors
  of the standard basis we have $(e_i . e_j) . e_k = e_i . (e_j .e_k)$.
  By~(\ref{law}), the left-hand side is equal to:
  $$\sum_l T_{ijl} (e_l.e_k) = \sum_l T_{ijl} \left(\sum_m T_{lkm} e_m\right) =
  \sum_m \left(\sum_l T_{ijl} T_{lkm}\right) e_m.$$
  A similar computation shows that
  $$e_i. (e_j.e_k) = \sum_m \left(\sum_l T_{jkl} T_{ilm}\right) e_m.$$
  Therefore we have
  $$\sum_l T_{ijl} T_{lkm} = \sum_l T_{jkl} T_{ilm}$$
  for every $m$. By symmetry of $T$, the left hand side is equal to
  $(T_i T_k)_{jm}$ where $T_1,\ldots,T_n$ denote the slices of $T$.
  As to the right-hand side, it is equal to $(T_k T_i)_{jm}$.
  Hence we have shown that $T_i T_k = T_k T_i$ for any $i,k$.

  Conversely, assume now that the slices $T_1,\ldots,T_n$ commute.
  The above computation shows that $(e_i . e_j) . e_k = e_i . (e_j .e_k)$, i.e.,
  associativity holds for basis vectors. The associativity of $(K^n,.)$
  then follows from bilinearity.
\end{proof}

\subsection{Orthogonal Waring decomposition over the complex numbers}

\label{csection}

Theorem \ref{realth} does not carry over directly to the field of complex numbers because complex symmetric matrices are not always diagonalizable:
consider for instance the matrices
$A=\begin{pmatrix}
  2i & 1\\
  1 & 0
\end{pmatrix}$
or $B=\begin{pmatrix}
  1 & i\\
  i & -1
\end{pmatrix}.$
The second matrix is not diagonalizable since $B^2=0$ but $B \neq 0$.
For the first one we have $(A-i\mathrm{Id})^2=0$ but $A \neq i \mathrm{Id}$.
\begin{theorem} \label{complexth}
A complex symmetric tensor of order 3
admits an orthogonal Waring decomposition if and only if its slices
are diagonalizable and pairwise commute.
 \end{theorem}
\begin{proof}
  Again we consider first the case where the symmetric tensor
  admits an orthogonal Waring decomposition.
The only difference with the real case is that we do need to show
that the 
slices are diagonalizable 
since this property does   not hold true for all complex symmetric matrices.
For the slices, this property follows from~(\ref{simslice}).

For the converse we note that since the slices commute and are now assumed to be diagonalizable, 
they are simultaneously diagonalizable as in the real case. We can therefore apply the complex case of Theorem~\ref{th:simdiag}
and conclude as in the proof of Theorem \ref{realth}.
\end{proof}
Recall that we denote by $\OW_n(\cc)$ the set of polynomials of $\cc[x_1,\ldots,x_n]_3$ that admit an orthogonal Waring decomposition.
As explained in Section~\ref{constructible},
Theorem~\ref{complexth} gives a description of $\OW_n(\cc)$ as a constructible
subset of $\cc[x_1,\ldots,x_n]_3$ (i.e. as a subset defined by a Boolean
combination of polynomial equalities).

We propose the following adaptation of the notion of a
{\em symmetrically  odeco} tensor
(\cite{boralevi17} and Proposition~\ref{symdefequiv}) 
from $K=\rr$ to $K=\cc$.
\begin{definition} \label{sodeco}
We denote by $\SODECO_n(\cc)$ the set of symmetric tensors of order 3 that
    can be written as $\sum_{i=1}^k  v_i^{\otimes 3}$
    where $v_1,\ldots,v_k$ are linearly independent pairwise orthogonal vectors
    in $\cc^n$.
    We use the same notation for the corresponding set of of degree 3
    homogenous polynomials in $\cc[x_1,\ldots,x_n]$.
\end{definition}
Over $\rr$ the linear independence would follow from the orthogonality
of the $v_i$ (compare with property (ii) in Proposition~\ref{symdefequiv}).
We need to add it explicitly in this definition
since some of the $v_i$ could be isotropic (a situation where pairwise orthogonality does not automatically imply linear independence).
Here is a characterization of $\OW_n(\cc)$ in the style
of Definition~\ref{sodeco}:
\begin{proposition} \label{snonisotropic}
  $OW_n(\cc)$ is equal to the set
  of homogeneous polynomials of degree 3 which admit a decomposition
  of the form
  $$f(x)=\sum_{j=1}^k  \langle u_j,x \rangle^3$$ for some $k \leq n$,
    where  $u_1,\ldots,u_k$
    are pairwise orthogonal non-isotropic vectors of $\cc^n$.
\end{proposition}
\begin{proof}
  It is very similar to the proof of Proposition~\ref{symdefequiv},
  and applies to homogeneous polynomials of any degree $d$.
  Suppose indeed that $f(x)=g(Ax)$ where $A$ is an orthogonal matrix and $g$ is
  as in~(\ref{powersum}).
  We saw that $f(x)=\sum_{j=1}^n a_j \langle u_j, x \rangle^d$
  where the $u_j$ are the rows of $A$. These vectors are indeed pairwise
  orthogonal and non-isotropic since they are the rows of an orthogonal matrix.
  We obtain the required decomposition by dropping the terms with $a_j=0$,
  and dividing the remaining $u_j$ by a $d$-th root of $a_j$.

  Conversely, assume that we have a decomposition
  $f(x)=\sum_{j=1}^k  \langle u_j,x \rangle^3$
    where  $u_1,\ldots,u_k$
    are pairwise orthogonal non-isotropic vectors of $\cc^n$.
    We can normalize these vectors and then extend $u_1,\ldots,u_k$ to
    an orthonormal basis of $\cc^n$ using Witt's extension theorem
    (Theorem~\ref{witt} and Corollary~\ref{wittcor}).
    This shows that $f(x)=g(Ax)$ where $A$ is an orthogonal matrix and $g$ is
  as in~(\ref{powersum}).
\end{proof}

\begin{corollary} \label{owinsodeco}
  We have $\OW_n(\cc) \subseteq \SODECO_n(\cc)$ for every $n \geq 1$.
This inclusion is strict for every $n \geq 2$.
\end{corollary}
\begin{proof}
  The inclusion is immediate from Definition~\ref{sodeco}
  and Proposition~\ref{snonisotropic}.
  In order to show that it is strict for $n=2$ consider the symmetric
  tensor $S=e^{\otimes 3}$ where $e=(1,i)$
  (or in an equivalent language,
  consider the polynomial $f(x_1,x_2)=(x_1+ix_2)^3$).
  The first slice of $S$ is the matrix $B$
  at the beginning of Section~\ref{csection},
  and it is not diagonalizable.
  Hence $S {\not \in} \OW_2(\cc)$ by Theorem \ref{complexth}, and it is clear
that $S \in \SODECO_2(\cc)$.

This example can be generalized to any $n \geq 2$ by adding ``dummy variables''
  $x_3,\ldots,x_n$. Namely, if we set $f(x_1,\ldots,x_n)=(x_1+ix_2)^3$,
  it is still the case that the matrix of the quadratic form
  $\partial f / \partial x_1$ is not diagonalizable.
  Hence $f {\not \in} \OW_n(\cc)$ by Theorem \ref{complexth}.
\end{proof}
In Section~\ref{sclosure} we will use the same example to show that  $\OW_n(\cc)$ is not closed for $n \geq 2$ (see Example~\ref{borderex} and Remark~\ref{rem:borderex}).

\subsection{Closure properties from below} \label{sclosure}

We will now study the closure $\overline{\OW_n(\cc)}$
of $\OW_n(\cc)$.
This study is justified by the fact that,
as we will see in Example~\ref{borderex}, $\OW_n(\cc)$
is not closed.
First, we show that $\OW_n(\cc)$ is ``large''
in the following sense.
\begin{proposition} \label{density}
  The set of first slices of symmetric tensors in $\OW_n(\cc)$
  is dense in the space of symmetric matrices of size $n$.
 \end{proposition}
An equivalent formulation in the langage of polynomials is that the
set of partial derivatives $\partial f / \partial x_1$ where
$f \in \OW_n(\cc)$ is dense in the space of quadratic forms in $n$
variables. Here the variable $x_1$ could of course be replaced by any other
variable.
\begin{proof}[Proof of Proposition~\ref{density}]
  By Lemma~\ref{diagsymdense}, it is sufficient to show that the
  set of diagonalizable symmetric matrices is in the closure of
  the set of first slices of symmetric tensors in $\OW_n(\cc)$.
  Consider therefore a diagonalizable symmetric matrix $S$ and the
  corresponding quadratic form $f(x)=x^TSx$.
  As recalled in Section~\ref{sec:simdiag},  there is
  an orthogonal matrix $A$ such that $A^TSA$ is diagonal.
  But $A^TSA$ is the matrix of the quadratic form $g(x)=f(Ax)$.
  Therefore, $f$ is of the form
  \begin{equation} \label{sumofsquares}
    f(x)=\sum_{i=1}^n a_i \langle u_i,x \rangle^2
    \end{equation}
    where the $u_i$ are the rows of the orthogonal matrix $A^T$.

  We first consider the case where the coefficients $u_{i1}$ of $x_1$
  in the linear forms $\langle u_i,x \rangle$ are all nonzero.
  In this first case, $S$ can be obtained exactly as the first slice
  of a symmetric tensor in $\OW_n(\cc)$. This is easily seen
  by integrating~(\ref{sumofsquares}) 
  with respect to $x_1$
  to obtain:
  $$F(x)=\sum_{i=1}^n \frac{a_i}{3u_{i1}} \langle u_i,x \rangle^3.$$
  This  polynomial is in $\OW_n(\cc)$, and $f= \partial F / \partial x_1$
  by construction.

  In the general case, by Lemma~\ref{approx_nonzero} below
  we can approximate the tuple
  of vectors $(u_1,\ldots,u_n)$ by a sequence of tuples
  $(v_1(\epsilon),\ldots,v_n(\epsilon))$ of pairwise orthogonal
  unit vectors such that $u_i = \lim_{\epsilon \rightarrow 0} v_i(\epsilon)$
  and $v_{i1}(\epsilon) \neq 0$ for all $i$.
  Our analysis of the first case shows that the quadratic forms
  $$f_{\epsilon}(x)=\sum_{i=1}^n a_i \langle v_i(\epsilon),x \rangle^2$$
  can be obtained exactly as partial derivatives
  $\partial F_{\epsilon}/\partial x_1$ with $F_{\epsilon} \in \OW_n(\cc)$.
    This completes the proof since $f=\lim_{\epsilon \rightarrow 0} f_{\epsilon}$.
\end{proof}

\begin{lemma} \label{approx_nonzero}
  Any complex orthogonal matrix can be approximated to an arbitrary precision
  by orthogonal matrices that have
    no vanishing entries in the first column.
  \end{lemma}
Here, ``approximated to an arbitrary precision''
means that given an orthogonal matrix $A$, we can construct for any small
enough $\epsilon$ an orthogonal matrix $A(\epsilon)$ with the required
property (having no vanishing entry in its first column) such that
$A = \lim_{\epsilon \rightarrow 0} A(\epsilon)$. Or in other words:
the set of orthogonal matrices is equal to the closure of the set of orthogonal matrices that have no vanishing entries in their first columns. 
\begin{proof}
  Let $A$ be an orthogonal matrix and let $c_1,\ldots,c_n$
  be its column vectors. Given $\epsilon \neq 0$ we construct the
  first column $c_1(\epsilon)$ of $A(\epsilon)$ by replacing any null entry
  of $c_1$ by $\epsilon$, and normalizing the resulting vector.
  Then we construct the remaining columns $c_2(\epsilon),\ldots,c_n(\epsilon)$
  by applying the Gram-Schmidt orthonormalization process to
  $c_1(\epsilon),c_2,\ldots,c_n$.
  For $\epsilon$ small enough,
  this process will yield pairwise orthogonal vectors
  $c_1(\epsilon),c_2(\epsilon),\ldots,c_n(\epsilon)$ of unit length
  ($\langle c_j(\epsilon),c_j(\epsilon) \rangle =1$)
  such that $\lim_{\epsilon \rightarrow 0} c_j(\epsilon) = c_j$ for all $j$.
  We therefore obtain an orthogonal matrix $A(\epsilon)$ such that
  $A = \lim_{\epsilon \rightarrow 0} A(\epsilon)$, and by construction
  the entries of the first column of $A(\epsilon)$ are all nonzero.
  
  In order to properly justify this construction it is important to note that,
  contrary to the case of real vectors,\footnote{For a reminder on the real
    or complex Hermitian case, see e.g.\cite[Section~0.6.4]{horn13}.}
  the Gram-Schmidt
  process cannot be carried out on all lists of linearly independent
  complex vectors.
  This is due to the existence of isotropic vectors in $\cc^n$ (such vectors
  cannot be normalized). In the present situation this process {\em will}
  nonetheless succeed for any small enough~$\epsilon$ due to the property
  $\lim_{\epsilon \rightarrow 0} c_j(\epsilon) = c_j$, which can be proved
  by induction on $j$. Since $\langle c_j,c_j \rangle =1$
  for all $j$, this property guarantees that we will not attempt
  to normalize any isotropic vector.
\end{proof}

We now give an example showing that $\OW_2(\cc)$ is not closed.
It can be easily generalized to any $n \geq 2$ by adding "dummy variables" like in the proof of Proposition~\ref {owinsodeco}.
\begin{example} \label{borderex}
The polynomial $f(x_1,x_2)=(x_1+ix_2)^3$ is in $\overline{\OW_2(\cc)} \setminus \OW_2(\cc)$.
One can show that  $f {\not \in}  \OW_2(\cc)$ using Theorem \ref{complexth}; this can be traced to the fact that  the vector $(1, i)$ is isotropic.
Indeed, the Hessian matrices of $\partial f / \partial x_1$ (respectively,  $\partial f / \partial x_2$) are:
$$6\begin{pmatrix}
  1 & i\\
  i & -1
\end{pmatrix},
6i\begin{pmatrix}
  1 & i\\
  i & -1
\end{pmatrix}.$$
These two matrices commute but they are not diagonalizable (they are nilpotent but nonzero).
To see that $f  \in  \overline{\OW_2(\cc)}$, consider the family of polynomials
\begin{equation} \label{limitexample}
  f_{\epsilon}(x_1,x_2) = [x_1+(i+\epsilon) x_2]^3 + \epsilon [(1+\epsilon^2)x_1+(i-\epsilon)x_2]^3.
\end{equation}
Note that $f = \lim_{\epsilon \rightarrow 0} f_{\epsilon}$.
It therefore remains to show that $f_{\epsilon} \in \OW_2(\cc)$
for all $\epsilon$ sufficiently close to 0.
This follows from the fact that the the vector $(1, i+\epsilon)$ is orthogonal 
to $(1+\epsilon^2,i-\epsilon)$.
Indeed, we have $f_{\epsilon} = g_{\epsilon}(A_{\epsilon}x)$
where $g_{\epsilon}(x_1,x_2)=x_1^3+\epsilon x_2^3$
and $$A_{\epsilon}=\begin{pmatrix}
  1 & i+\epsilon\\
  1+\epsilon^2 & i-\epsilon
\end{pmatrix}.$$
We just pointed out that the rows of $A_{\epsilon}$ are orthogonal, and for $\epsilon$ nonzero but small enough
they can be normalized since they are not self-orthogonal.
We can therefore write $A_{\epsilon} = D_{\epsilon}U_{\epsilon}$ where $D_{\epsilon}$
is a diagonal matrix and $U_{\epsilon}$ orthogonal.
Hence we have
\begin{equation} \label{symepsilondec}
  f_{\epsilon}=g_{\epsilon}(D_{\epsilon}U_{\epsilon}x)=g'_{\epsilon}(U_{\epsilon})
  \end{equation}
where $g'_{\epsilon}=\alpha_{\epsilon}x_1^3+\beta_{\epsilon}x_2^3$
for some appropriate coefficients $\alpha_{\epsilon}, \beta_{\epsilon}$.
We conclude that~(\ref{symepsilondec}) provides as needed
an orthogonal decomposition of $f_{\epsilon}$.
\end{example}
\begin{remark} \label{rem:borderex}
  For an even simpler family of polynomials in $\OW_2(\cc)$ witnessing the
  fact that $(x_1+ix_2)^3 \in \overline{\OW_2(\cc)}$, one may replace the
  coefficient~$\epsilon$ of $[(1+\epsilon^2)x_1+(i-\epsilon)x_2]^3$
  in~(\ref{limitexample}) by 0. We obtain the family of polynomials
  $f'_{\epsilon}(x_1,x_2) = [x_1+(i+\epsilon) x_2]^3$, which also has 
  $(x_1+ix_2)^3$ as its limit.
  Moreover, the same argument as in Example~\ref{borderex} shows that
  $f'_{\epsilon} \in \OW_2(\cc)$ for any small enough nonzero $\epsilon$.
  
  Alternatively, one can use Theorem \ref{complexth} to show that $f'_{\epsilon} \in \OW_2(\cc)$. 
  The slices of $f'_{\epsilon}$ are 
  $$S_{\epsilon,1}=\begin{pmatrix}
  1 & i+\epsilon\\
  i+\epsilon& (i+\epsilon)^2
\end{pmatrix},
S_{\epsilon,2}=(i+\epsilon)S_{\epsilon,1}.$$ 
These matrices commute. Moreover, for any small enough nonzero $\epsilon$ the characteristic polynomial $\lambda^2-\lambda[(I+\epsilon)^2+1]$ of $S_{\epsilon,1}$ has distinct roots and the two slices will  therefore be diagonalizable.
\end{remark}
We will now give an example of a polynomial $g \in \overline{\OW_2(\cc)} \setminus \OW_2(\cc)$ for which this property is more delicate to establish than
in the previous example.
\begin{example} \label{bordersodeco}
  Let $g(x_1,x_2)=x_2(x_1+ix_2)^2$. In order to show that $g \in \overline{\OW_2(\cc)}$, one can check
  (by a tedious calculation or the help of a computer algebra system)
  that $g(x) = \lim_{\epsilon \rightarrow 0} g_{\epsilon}(A_{\epsilon} x)$
  where $A_{\epsilon}$ is the matrix of Example~\ref{borderex}
  and
  $$g_{\epsilon}(x_1,x_2)=\frac{(1+\epsilon^2)^3x_1^3-x_2^3}{6\epsilon}.$$

  The Hessian matrix of $\partial g / \partial x_2$ is equal to $2M$
  where $$M=\begin{pmatrix}
  1 & 2i\\
  2i & -3
  \end{pmatrix}.$$
  Since $-1$ is the only eigenvalue of $M$ but
  $M \neq -\mathrm{Id}$, this matrix is not diagonalizable.
  Like in Example~\ref{borderex},
  Theorem \ref{complexth} therefore implies
  that $g {\not \in} \OW_2(\cc)$.
\end{example}
For the polynomial $g$ in the above example, we will now prove
a stronger result than $g {\not \in} \OW_2(\cc)$.
Indeed, it follows from the next lemma that $g$ cannot be expressed
as the sum of cubes of two linear forms.
\begin{lemma} \label{lem:2cubes}
  Let $h(x_1,x_2)=l_1(x_1,x_2)l_2(x_1,x_2)^2$ where
  $l_1,l_2$ are two linear forms.
  If there are two linear forms $\ell_1,\ell_2$ such that
  $$h(x_1,x_2)=\ell_1(x_1,x_2)^3+\ell_2(x_1,x_2)^3$$
  then $l_1,l_2$ cannot be linearly independent.
\end{lemma}
\begin{proof}
  Suppose that $l_1,l_2$ are linearly independent. By performing an invertible
  change of variables we can reduce to the case $l_1=x_1,l_2=x_2$.
  We must therefore show that one cannot write
  $$x_1x_2^2 = \ell_1(x_1,x_2)^3+\ell_2(x_1,x_2)^3.$$
  If such a decomposition exists, the linear forms $\ell_1,\ell_2$ must be linearly independent
  since $x_1x_2^2$ is not the cube of a linear form.
  By Lemma~\ref{lem:hessian}, the Hessian determinant of the right hand side
  must be a constant multiple of $\ell_1 \ell_2$. In particular, it is
  a squarefree polynomial. 
  But a simple computation shows that the Hessian determinant 
  of $x_1x_2^2$ is equal to $-4x_2^2$.
\end{proof}
The polynomial of Example~\ref{bordersodeco}
can be viewed as a polynomial in $n>2$ variables
by introducing $n-2$ dummy variables,
i.e., we can set $g(x_1,x_2,x_3,\ldots,x_n)=x_2(x_1+ix_2)^2$.
Next we show that introducing these extra variables does not help write
$g$ as a sum of cubes of linearly independent linear forms.
In fact, like Lemma~\ref{lem:2cubes} this result holds for any polynomial
of the form $l_1(x_1,x_2)l_2(x_1,x_2)^2$ where the linear forms $l_1,l_2$
are linearly independent.
\begin{proposition} \label{prop:x_1x_2^2}
  Suppose that
  $$l_1(x_1,x_2)l_2(x_1,x_2)^2 =
  \ell_1(x_1,\ldots,x_n)^3+\cdots+\ell_k(x_1,\ldots,x_n)^3.$$
  If the linear forms $l_1,l_2$ are linearly independent then the
  linear forms $\ell_1,\ldots,\ell_k$ cannot be linearly independent.
\end{proposition}
\begin{proof}
  By performing an invertible change of variables we can assume without loss
  of generality that $l_1=x_1$ and $l_2=x_2$.
  Let us therefore suppose by contradiction that
  \begin{equation} \label{x_1x_2^2}
    x_1 x_2^2 =
    \ell_1(x_1,\ldots,x_n)^3+\cdots+\ell_k(x_1,\ldots,x_n)^3
    \end{equation}
 where the $\ell_i$ are linearly independent.
 We consider first the case $k=n$.
By Lemma~\ref{lem:2cubes} we must have $n \geq 3$. 
By Lemma~\ref{lem:hessian} the Hessian determinant of the
 right-hand side of~(\ref{x_1x_2^2}) is not identically zero;
 but for $n \geq 3$, the Hessian determinant of the left hand side (viewed as a polynomial in $n$ variables) is~0 since $x_1x_2^2$ does not depend on the last $n-2$ variables.
 
 It therefore remains to consider the case $k <n$. Since the $\ell_i$ are linearly independent, one can set $n-k$ of the variables $x_1,\ldots,x_n$ to~0
 so that the resulting linear forms $\ell_1',\ldots,\ell_k'$ in $k$ variables
 remain linearly independent.
 If $x_1$ or $x_2$ are among the variables that have been set to~0, we obtain
 from~(\ref{x_1x_2^2}) the identity $\ell_1'^3+\cdots+\ell_k'^3=0$.
 This is impossible since the Hessian determinant of the left hand side
 is nonzero by Lemma~\ref{lem:hessian}.
 If $x_1,x_2$ have not been set to~0 we obtain 
 $$x_1x_2^2 = \ell_1'^3+\cdots+\ell_k'^3$$
 and we are back to the case $k=n$: as shown earlier in the proof,
 this identity cannot be satisfied if the $\ell_i'$ are linearly independent.
\end{proof}
In the remainder of Section~\ref{sclosure} we prove the central result
of this section:\footnote{The theorem's statement already appears as Theorem~\ref{firstinclosure} before Section~\ref{realwaring}. We reproduce it here for
  the reader's convenience.}
\begin{theorem} \label{inclosure}
  For every $n \geq 1$ we have $$\OW_n(\cc) \subseteq \SODECO_n(\cc) \subseteq \overline{\OW_n(\cc)}.$$
  These two inclusions are strict for every $n \geq 2$.
  \end{theorem}
For the proof 
we need some elements of the theory of quadratic forms,
see Section~\ref{quadratic}.

\begin{lemma} \label{approx}
  Let $(u_1,\ldots,u_k)$ be a tuple of linearly independent pairwise orthogonal
  vectors of $\cc^n$. This tuple can be approximated to an arbitrary precision
  by tuples $(v_1,\ldots,v_k)$ of pairwise orthogonal vectors with
  $\langle v_j, v_j \rangle \neq 0$ for all $j=1,\ldots,k$.
\end{lemma}
Here, ``approximated to an arbitrary precision''
means that for every $\epsilon > 0$ there is a tuple
$(v_1(\epsilon),\ldots,v_k(\epsilon))$ of pairwise orthogonal 
non-isotropic vectors such that $u_j=\lim_{\epsilon \rightarrow 0} v_j(\epsilon)$ for all $j=1,\ldots,k$.
Note that the orthogonality conditions on the $v_j$ together with 
the conditions $\langle v_j, v_j \rangle \neq 0$
imply that these vectors are linearly independent.
\begin{proof}
  Let us assume that the isotropic vectors appear first in 
  $(u_1,\ldots,u_n)$, i.e., $u_1,\ldots,u_r$ are
  (for some $r \geq 0$) the isotropic vectors in this list.
  We denote by $U$ the subspace spanned by the isotropic vectors, and by $V$
  the subspace spanned by $(u_{r+1},\ldots,u_k)$.
  By Corollary~\ref{orthosupplement} we can choose for $U$ an isotropic
  supplement $U'$ of $U$ which is orthogonal to $V$.
  Let $(u'_1,\ldots,u'_r)$ be a basis of $U'$ which satisfies the property
  $\langle u_i,u'_j \rangle = \delta_{ij}$ from Theorem~\ref{supplement}.

Given $\epsilon$ we approximate $u_j$ by
$v_j(\epsilon)=u_j+\epsilon u'_j$ for $j \leq r$.
We set $v_j(\epsilon) = u_j$ for $r+1 \leq j \leq k$.
Obviously, we have $u_j=\lim_{\epsilon \rightarrow 0} v_j(\epsilon)$ for all $j$.

For $j \leq r$ and $l > r$ we have
$$\langle v_j(\epsilon),v_l(\epsilon) \rangle =
\langle u_j+\epsilon u'_j, u_l \rangle =0$$
since $\langle u_j, u_l \rangle = 0$ ($u_1,\ldots,u_k$ are pairwise orthogonal)
and $\langle u'_j, u_l \rangle = 0$ ($U'$ is orthogonal to $V$).
For $j,l >r$ we have $\langle v_j(\epsilon),v_l(\epsilon) \rangle =
\langle u_j, u_l \rangle$; this is nonzero for $j=l$, and 0 for $j \neq l$.
For $j,l \leq r$ 
we have $$\langle v_j(\epsilon),v_l(\epsilon) \rangle =
\langle u_j+\epsilon u'_j, u_l+\epsilon u'_l \rangle =
\epsilon[\langle u_j, u'_l \rangle + \langle u'_j, u_l \rangle]$$
since $U$ and $U'$ are totally isotropic.
This is equal to $2\epsilon$ for $j=l$, and to 0 for $j \neq l$.
 \end{proof}

\begin{proof}[Proof of Theorem~\ref{inclosure}]
  The first inclusion, and the fact it is strict for $n\geq 2$, is given
  by Corollary~\ref{owinsodeco}.
  For the second inclusion,
  consider a symmetric tensor $S \in \SODECO_n(\cc)$ and the corresponding decomposition $S=\sum_{j=1}^k  u_j^{\otimes 3}$ from Definition~\ref{sodeco}.
  Lemma~\ref{approx} provides us with pairwise orthogonal non-isotropic vectors
  $v_j(\epsilon)$ such that $u_j = \lim_{\epsilon \rightarrow 0} v_j(\epsilon)$.
  As a result, $S$ is the limit of the symmetric tensors
  $S(\epsilon)=\sum_{j=1}^k  v_j(\epsilon)^{\otimes 3}$.
  These tensors are in $\OW_n(\cc)$ by Proposition~\ref{snonisotropic}.

  It remains to show that the second inclusion is strict for $n \geq 2$.
  The witness of this separation will be the polynomial
  $$g_n(x_1,\ldots,x_n) = x_2(x_1+ix_2)^2.$$
  We have shown in Example~\ref{bordersodeco}
  that $g_2 \in \overline{\OW_2(\cc)}$.
  For $n>2$ the property $g_n \in \overline{\OW_n(\cc)}$ continues to hold 
  since $\OW_2(\cc) \subseteq \OW_n(\cc)$.\footnote{For this inclusion to make
    sense we must of course continue to view a polynomial in two variables
    as a polynomial in $n$ variables.}
  Finally, Proposition~\ref{prop:x_1x_2^2}
  shows that $g_n {\not \in} \SODECO_n(\cc)$.
  Note that the polynomial of Example~\ref{borderex} cannot be used
  instead of $g_n$ to prove this separation
  since it belongs to $\SODECO_n(\cc)$.
\end{proof}

\subsection{The ASD property} \label{sec:asd}
  
In this section we investigate $\overline{\OW_n(\cc)}$ "from above", i.e., we find properties that must be satisfied by all
of its elements. Here is a simple example of such a property:
\begin{proposition} \label{commutativity}
  If $f \in \overline{\OW_n(\cc)}$ the matrices $A_1,\ldots,A_n$ in
  (\ref{partial3}) pairwise commute.
\end{proposition}
\begin{proof}
  This follows from Theorem~\ref{complexth} and the fact that the commutativity
  conditions $A_iA_j-A_jA_i=0$ are closed.
\end{proof}
Are there additional conditions that must be satisfied by the polynomials
in $\overline{\OW_n(\cc)}$ besides the above commutativity conditions?
The {\em ASD property} defined in the introduction turns out to be useful
for the investigation of this question.
There is a significant body of work on the ASD property,
see~\cite{omeara06,omeara11} and the references therein.
\begin{proposition} \label{owasd}
  If $f \in \overline{\OW_n(\cc)}$ the matrices $A_1,\ldots,A_n$ in (\ref{partial3}) are approximately simultaneously diagonalizable.
\end{proposition}
\begin{proof}
  This follows immediately from Theorem~\ref{complexth} and the already mentioned fact that a collection of matrices is simultaneously diagonalizable
  iff each matrix is diagonalizable and they pairwise commute.
\end{proof}
It is not clear whether the converse of this proposition holds because
a tuple $(B_1,\ldots,B_n)$ approximating $(A_1,\ldots,A_n)$ in the sense of
Definition~\ref{asd} might not come from a symmetric tensor.
In fact, it is not clear whether a tuple of symmetric matrices
satisfying the ASD property can always be approximated in the sense of
Definition~\ref{asd} by a tuple of simultaneously diagonalizable
{\em symmetric} matrices.

In the next theorem we collect some properties that must be satisfied
by any tuple of matrices satisfying the ASD property.
\begin{theorem} \label{ASDimplies}
 Any tuple $(A_1,\ldots,A_k)$  of approximately simultaneously diagonalizable matrices of $M_n(\cc)$
  must satisfy the following properties:
 \begin{itemize}
 \item[(i)] The matrices $A_1,\ldots,A_k$ pairwise commute.
 \item[(ii)] The subalgebra of $M_n(\cc)$ generated by $A_1,\ldots,A_k$ and the identity matrix is of dimension at most $n$.
 \item[(iii)] The centralizer of $A_1,\ldots,A_k$ is of dimension at least $n$.
 \end{itemize}
  \end{theorem}
  The first property can be found in \cite{omeara06} and follows easily from the fact (already used in the proof of Proposition \ref{commutativity}) that the commutativity of two matrices is a closed condition.
  The second property is established in the same paper, and the third one is Theorem 6.6.2 from \cite{omeara11}.
    From Theorem~\ref{ASDimplies} and Proposition \ref{owasd} we have:
  \begin{theorem} \label{owbar}
  If $f \in \overline{\OW_n(\cc)}$ the matrices $A_1,\ldots,A_n$ in (\ref{partial3}) must satisfy the following properties:
  \begin{itemize}
 \item[(i)]  $A_iA_j = A_j A_i$ for all $i,j$.
 \item[(ii)] The subalgebra of $M_n(\cc)$ generated by $A_1,\ldots,A_n$ and the identity matrix is of dimension at most $n$.
 \item[(iii)] The centralizer of $A_1,\ldots,A_n$ is of dimension at least $n$.
 \end{itemize}
  \end{theorem}
  In the remainder of this section we show that 
  the converse of this theorem does not hold:
  the conjunction of these 3 conditions does not  imply membership
  in $\overline{\OW_n(\cc)}$.

\begin{theorem} \label{outclosure}
    For every $n \geq 21$ there is a polynomial $f \in \cc[X_1,\ldots,X_n]_3$
    which satisfies conditions
    (i), (ii) and (iii) of Theorem~\ref{owbar} but does not belong to
    $\overline{\OW_n(\cc)}$.
\end{theorem}
  It is already known that the 3  properties of Theorem~\ref{ASDimplies}
together still do not suffice to imply the ASD
property~\cite{deboor08,omeara11}.
Theorem~\ref{outclosure} does not follow from this result
since the counterexamples
in~\cite{deboor08,omeara11} are not constructed from symmetric tensors.

Before giving the proof of Theorem~\ref{outclosure}
we need some preliminary results.
  In particular, we need the following lemma about the Waring rank of forms
  of degree 3 (recall that the Waring rank of a homogeneous polynomial
  $f$ of degree $d$ is the smallest $r$ such that $f$ can be written as
  a sum of $r$ $d$-th powers of linear forms).
\begin{lemma} \label{ah}
  If $v,r$ are two integers with $r \leq \lceil (v+1)(v+2)/6 \rceil$,
  there exists a degree 3 homogeneous polynomial in $v$ variables
  which can be written as
  a sum of $r$ cubes of linear forms,  but not as a sum of fewer cubes.
\end{lemma}
\begin{proof}
  For $v \geq 2$, let $w(v)$ be the maximum Waring rank
  of any homogeneous polynomial of degree 3 in $v$ variables.
 We claim that for every $r \leq w(v)$  there exists a degree 3 homogeneous polynomial
  in~$v$ variables which can be written as a sum of $r$ cubes of linear forms,  but not as a sum of fewer cubes. This is easily proved by downward induction on~$r$. Indeed, this property holds true for $r=w(v)$ by definition of
  the maximum rank. For the induction step, assume by contradiction
  that for some $r < w(v)$ every sum of $r$ cubes can be rewritten
  as a sum of $r-1$ cubes. Then we could also rewrite every sum of $r+1$ cubes
  as a sum of $r$ cubes (by rewriting the sum of the first $r$ cubes as a sum
  of $r-1$ cubes).

  Thus it remains to show that $w(v) \geq \lceil (v+1)(v+2)/6 \rceil$.
  This follows from a simple dimension count. Indeed, $\cc[x_1,\ldots,x_v]_3$
  is of dimension $\binom{v+2}{3}$ but the set of polynomials 
   that can be expressed as a sum of $r$ cubes
   is of dimension\footnote{as a constructible subset
     of  $\cc[x_1,\ldots,x_v]_3$} at most $rv$ since we have $v$ parameters (the coefficients of the corresponding linear function) for each cube.
  Hence $w(v) \geq \binom{v+2}{3} / v$.
\end{proof}
\begin{remark}
  The above dimension count shows 
  that  $\lceil (v+1)(v+2)/6 \rceil$
  is a lower bound on the Waring rank of a {\em generic} polynomial
  of $\cc[x_1,\ldots,x_v]_3$. Moreover, it is known (but much harder to prove)
  that this is
  the exact value of the generic Waring rank, except for
  $v=5$ where the generic rank is 8 instead of 7.
  This is the degree 3 case of the celebrated
  Alexander-Hirschowitz theorem~\cite{alexander95,brambilla08}, which
  determines the exact value of the generic Waring rank for any number
  of variables and any degree.
  The exact value of the maximum Waring rank does not seem to be known in
  general, but it is never more than
  twice the generic rank~\cite{blekherman15}.
\end{remark}
For the proof of Theorem~\ref{outclosure}
we also need the notion of {\em border Waring rank}.
\begin{definition}
  A polynomial $f \in \cc[X_1,\ldots,X_n]_3$ is of border Waring rank
  at most $r$ if there is a sequence $(f_k)_{k \geq 0}$ of polynomials
  of $\cc[X_1,\ldots,X_n]_3$ such that $f=\lim_{k \rightarrow +\infty} f_k$
   and each $f_k$ is of Waring rank at most $r$.
\end{definition}
In other words, $f$ is of border Waring rank
at most $r$ if it is in the closure of the set of polynomials of Waring rank
at most $r$.
\begin{lemma} \label{highborder}
  For any integer $v$ there exists a polynomial $f \in \cc[X_1,\ldots,X_v]_3$
  of border Waring rank at least $\lceil (v+1)(v+2)/6 \rceil$.
\end{lemma}
\begin{proof}
  We have seen in the proof of Lemma~\ref{ah} that if $r <  (v+1)(v+2)/6$,
  the set of polynomials of Waring rank at most $r$ is of dimension
  less than $\dim \cc[X_1,\ldots,X_v]_3$.
  Taking the closure does not increase the dimension, so the closure remains
  of dimension less than $\dim \cc[X_1,\ldots,X_v]_3$ and there must
  exist a polynomial in its complement.
  \end{proof}
  The following remark will be useful for the proof of Theorem \ref{outclosure}.
  \begin{remark} \label{morevars}
  For any $n > w$ we can view any polynomial $f  \in \cc[X_1,\ldots,X_w]_3$  as a polynomial in $ \cc[X_1,\ldots,X_n]_3$ which does not depend on its last $n-w$ variables, and this does not change the border Waring rank of $f$. 
  Let us indeed denote by $r_w$ (respectively, $r_n$) the border Waring rank of $f$ when viewed as a polynomial in $w$ (respectively, $n$) variables.
  By definition of $r_w$ there is a sequence  $(f_k)_{k \geq 0}$ of polynomials in $ \cc[X_1,\ldots,X_w]_3$ 
  of Waring rank at most~$r_w$ such that $f=\lim_{k \rightarrow +\infty} f_k$. The $f_k$ can also be viewed as polynomials in $n$ variables, hence $r_n \leq r_w$. 
Likewise,  there is a sequence $(g_k)_{k \geq 0}$ of polynomials in $ \cc[X_1,\ldots,X_n]_3$ 
  of Waring rank at most $r_n$ such that $f=\lim_{k \rightarrow +\infty} g_k$. 
  The polynomials $f'_k(X_1,\ldots,X_w) = g_k(X_1,\ldots,X_v,0,\ldots,0)$ are of Waring rank at most $r_n$, and  $f=\lim_{k \rightarrow +\infty} f'_k$. Hence $r_w \leq r_n$.
  \end{remark}
 Finally, we need two results involving isotropic vectors.
\begin{proposition} \label{zeroproduct}
  If a symmetric tensor $S$ admits a decomposition of the form
  $$S=\sum_{i=1}^r u_i^{\otimes 3}$$ where the $u_i$
  are pairwise orthogonal nonzero vectors then the slices of $S$
  commute. Moreover, for such a decomposition the two following properties
  are equivalent:
  \begin{itemize}
  \item[(i)] All the $u_i$ are isotropic.
  \item[(ii)] The product of any two (possibly equal) slices of $S$
    is equal to 0.
  \end{itemize}
  As a result if (i) holds then in any other decomposition
  $S=\sum_{i=1}^q v_i^{\otimes 3}$ where the $v_i$
  are pairwise orthogonal nonzero vectors, all the $v_i$ must be isotropic.
 \end{proposition}
\begin{proof}
  The $k$-th slice of a rank one symmetric tensor $u^{\otimes 3}$
  is the matrix $u_k u u^T$, where $u_k$ denotes the $k$-th component of $u$.
  For the tensor $S$ the $k$-th slice $S_k$
  is therefore equal to $\sum_{i=1}^r u_{ik}u_i u_i^T$ where $u_{ik}$ denotes
  the $k$-th component of $u_i$. The product of two slices is therefore given
  by the formula:
  $$S_k S_l = \sum_{i,j = 1}^r u_{ik} u_{jl} u_iu_i^Tu_ju_j^T.$$
  In this expression the products $u_i^Tu_j$ vanish for $i \neq j$
  since the $u_i$ are pairwise
  orthogonal. As a result we obtain
  \begin{equation} \label{sliceproduct}
    S_k S_l = \sum_{j=1}^r u_{jk} u_{jl} (u_j^T u_j) (u_j u_j^T)
  \end{equation}
  and this expression is symmetric in $k$ and $l$.

  Next we show that (i) and (ii) are equivalent.
  Assume first that all the $u_i$ are isotropic. Then $S_k S_l=0$ since
  the coefficients $u_j^T u_j$ in~(\ref{sliceproduct}) are all equal to 0.
  Conversely, assume that some vector $u_i$ is not isotropic.
  Since $u_i \neq 0$, at least one component $u_{ik}$ is nonzero.
  We claim that $u_i$ is an eigenvector of $S_k^2$ associated to a nonzero
  eigenvalue, thereby showing that $S_k^2 \neq 0$.
  Indeed, from~(\ref{sliceproduct}) and from the orthogonality of the $u_j$
  we have
  \begin{equation} \label{eigen}
    S_k^2 u_i = u_{ik}^2 (u_i^T u_i)^2 u_i 
   \end{equation}
  and the eigenvalue $u_{ik}^2 (u_i^T u_i)^2$ is nonzero as promised.
  One can also obtain~(\ref{eigen}) from the observation that
  $u_i$ is an eigenvector of $S_k$ associated to the eigenvalue
  $u_{ik} (u_i^T u_i)$.
 The last part of the proposition is clear: if (i) holds then the product of any two slices of $S$ must be equal to 0 since (i) implies (ii). From the converse
  implication (ii) $\implies$ (i) applied to the decomposition
  $S=\sum_{i=1}^r v_i^{\otimes 3}$ it follows that the $v_i$ are isotropic.
\end{proof}

\begin{proposition} \label{centralizer}
  Let $(v_1,\ldots,v_s)$ be a basis of a totally isotropic subspace~$V$.
 The $s^2$ matrices $v_i v_j^T$ are linearly independent.
  Moreover, the product of any two matrices of the form $uv^T$ with $u,v \in V$
  is equal to 0 (as a result, these matrices span a commutative matrix algebra
  of dimension $s^2$).
  \end{proposition}
\begin{proof}
  Recall from Theorem \ref{supplement} that there are vectors $v'_1,\ldots,v'_s$ such that $v_j^Tv'_i= \delta_{ij}$ for all $i,j$
and assume that $\sum_{i,j} \alpha_{i,j} v_i v_j^T = 0$. Multiplying this identity by any of the  $v'_k$ on the right shows that $\sum_i \alpha_{i,k} v_i =0$. Since the $v_i$ are linearly independent, $\alpha_{1,k},\ldots,\alpha_{s,k}$ must all be equal to 0.
Since this is true for any $k$ we conclude that the $s^2$ matrices are indeed linearly independent.
Finally, the product of $uv^T$ and $u'v'^T$ where $u,u',v,v' \in V$ is equal to 0
  since $v^T u' = 0$.
\end{proof}

\begin{proof}[Proof of Theorem~\ref{outclosure}]
  Let us denote by $S$ the tensor of the polynomial $f$ to be constructed.
  We will choose $S$ of the form  $\sum_{i=1}^r u_i^{\otimes 3}$
  where the $u_i$ belong to a totally isotropic
  subspace of $\cc^n$.
  By Proposition~\ref{zeroproduct} the slices of such a tensor commute,
  so condition (i) will be satisfied.  Moreover, since the product of any two
  slices is equal to 0 the algebra in condition (ii) reduces to the
  linear subspace spanned by the slices of $S$ and the identity matrix.
  We have seen in the proof of Proposition~\ref{zeroproduct} that the slices
  of $S$ are given by the formula:
  $$S_k=\sum_{i=1}^r u_{ik}u_i u_i^T.$$
  The slices of $S$ are therefore linear combinations of the $r$ matrices
  $u_i u_i^T$, and the coefficients of the linear combinations are the entries
  of the matrix $(u_{ik})_{1 \leq i \leq r, 1 \leq k \leq n}$. This matrix is of
  rank at most $n/2$ since the $u_i$ lie in a totally isotropic subspace.
  As a result the slices span a space of dimension at most $n/2$.
  Taking the identity matrix into account, we conclude that the algebra
  is of dimension at most $1+n/2 \leq n$.
  Regarding condition (iii), observe that the centralizer of the slices $S_1,\ldots,S_n$ contains the centralizer of the matrices $u_i u_i^T$.
  Therefore, if we take the $u_i$ in a totally isotropic subspace of dimension $\lfloor n/2 \rfloor$ it follows from Proposition \ref{centralizer} that the centralizer will be of dimension at least $\lfloor n/2 \rfloor^2$.
  Hence condition (iii) will be satisfied for $n \geq 6$.

  It remains to choose the tensor $S$ so that $f {\not \in} \overline{OW_n(\cc)}$.  Let $v=\lfloor n/2 \rfloor$ and let  $g(x_1,x_3,x_5,\ldots,x_{2v-1})$ be a degree 3 form of border Waring rank at least  $\overline{r}=\lceil (v+1)(v+2)/6 \rceil$ (the existence of $g$ is guaranteed by Lemma~\ref{highborder}).
  Note that $\overline{r} > n$ if $n \geq 21$.
  Let \begin{equation} \label{fromgtof}
  f(x_1,x_2,x_3,\ldots,x_{2v-1},x_{2v})=
g(x_1+ix_2,x_3+ix_4,\ldots,x_{2v-1}+ix_{2v}).
\end{equation}
Note that $f$ and $g$ have same Waring rank and same border Waring rank.
In particular, the border Waring rank of $f$ is greater than $n$ for $n \geq 21$. This shows that $f {\not  \in} \overline{\OW_n(\cc)}$ since an orthogonal Waring decomposition is a particular Waring decomposition of size $n$
(for odd $n$ we apply Remark~\ref{morevars} to $f$ with $w=2v=n-1$).

Let $(e_1,\ldots,e_{2v})$ be the standard basis of $\cc^{2v}$ and let $V$
be the totally isotropic subspace spanned by
$e_1+ie_2,e_3+ie_4,\ldots,e_{2v-1}+ie_{2v}$.
Finally, let $S$ be the tensor of $f$ and let $r$ be the Waring rank of $g$.
From the corresponding decomposition of $g$ and from~(\ref{fromgtof})
we obtain a decomposition  $S = \sum_{i=1}^r u_i^{\otimes 3}$ where the $u_i$
belong to $V$.
Since this subspace is totally  isotropic,
as explained at the beginning of the proof
$f$ will satisfy conditions (i), (ii) and (iii).
\end{proof}

\begin{remark}
  Dimension arguments play an important role in the proof of
  Theorem~\ref{outclosure}, and they are captured by the notion of
  ``border Waring rank.'' These dimension arguments can be written in
  a more concise way without appealing explicitly to the notion of (border)
  rank; see the proof of Theorem~\ref{outclosureot} in the next section
  for the case of ordinary tensors.
\end{remark}

\section{Orthogonal decomposition of ordinary tensors} \label{ordi}

In this section we consider orthogonal decompositions of ordinary
(possibly non symmetric) tensors of order 3.
Compared to the symmetric case, the results are similar but their statements and proofs are somewhat more complicated.
First, we point out that 
for $K=\rr$ 
Definition~\ref{orthodef} agrees with the definition of an ``odeco tensor''
from~\cite{boralevi17}:
\begin{proposition} \label{waringtensor}
  Let $t \in \rr[x_1,\ldots,x_n]_3$ be a trilinear form
  and let $T$ be the corresponding tensor.
  The two following properties are equivalent:
  \begin{itemize}
  \item[(i)] $t$ admits an orthogonal decomposition.
  \item[(ii)] $T$ is odeco~\cite{boralevi17},
    i.e., can be written as
    $$\sum_{i=1}^k u_i \otimes v_i \otimes w_i$$
    where each of the the 3 lists $(u_1,\ldots,u_k)$, $(v_1,\ldots,v_k)$,
    $(w_1,\ldots,w_k)$ is made of $k$ nonzero, pairwise orthogonal vectors
    in $\rr^n$.
  \end{itemize}
\end{proposition}
This proposition could be extended to multilinear forms of degree $d$ and ordinary tensors of order $d$. We skip the proof, which is essentially the same
as for Proposition~\ref{symdefequiv}.
Like in the symmetric case the equivalence of (i) and (ii) fails over the field
of complex numbers due to the existence of isotropic vectors,
see Example~\ref{borderexordi} in Section~\ref{closure}.
Over $\cc$ we will therefore have again two different notions of
``orthogonal decomposition.'' The corresponding classes of tensors are denoted $\OT_n(\cc)$ and $\ODECO_n(\cc)$ (the latter defined in Section~\ref{closure}). These are the analogues of the classes $\OW_n(\cc)$ and $\SODECO_n(\cc)$ studied in Section~\ref{sec:waring}.

\begin{proposition} \label{slices}
  Let $f(x,y,z)=g(Ax,By,Cz)$ where $g$ is a trilinear form with slices
  $\frac{\partial g}{\partial z_i} = x^TS_iy$. The slices of $f$ are given
  by the formula: $$\frac{\partial f}{\partial z_k} = x^TT_ky$$ where
  $T_k = A^TD_kB$, $D_k=\sum_{i=1}^n c_{ik}S_i$ and the $c_{ik}$ are the entries
  of $C$.

  In particular, if $g$ is as in~(\ref{diagtensor})
  we have 
  $D_k = \diag(\alpha_1 c_{1k},\ldots,\alpha_n c_{nk})$.
\end{proposition}
\begin{proof}
  Differentiating the expression $f(x,y,z)=g(Ax,By,Cz)$
  shows that
  $$\frac{\partial f}{\partial z_k} =
  \sum_{i=1}^n c_{ik} \frac{\partial g}{\partial z_i}(Ax,By,Cz),$$
  and the result follows by plugging the expression
  $\frac{\partial g}{\partial z_i} = x^TS_iy$ into this formula.

  In the case where  $g$ is as in~(\ref{diagtensor}),
  $S_i$ is the diagonal matrix
  with an entry equal to $\alpha_i$ at row $i$ and column $i$,
  and zeroes elsewhere. As a result
  $D_k=\diag(\alpha_1 c_{1k},\ldots,\alpha_n c_{nk})$.
\end{proof}
Note that the above proof applies to arbitrary matrices $A,B,C$ (orthogonality is not used).
\begin{corollary} \label{cor:slices}
  If a tensor $T \in K^{n \times n \times n}$ admits an orthogonal decompostion,
  its $n$ slices $T_1,\ldots,T_n$ satisfy the following conditions:
  for $1 \leq k,l \leq n$, the matrices $T_k T_l^T$ (respectively, $T_k^T T_l$)
  are symmetric and pairwise commute.
  \end{corollary}
\begin{proof}
  Let us prove this for the $z$-slices (the same property of course holds
  for slices in the $x$ and $y$ directions) and for the matrices $T_k T_l^T$.
  
  By Proposition~\ref{slices}, $T_k = A^TD_kB$ where $D_k$ is diagonal.
  We have:
  $$T_k T_l^{T}=(A^TD_kB)(B^TD_lA) = A^TD_k D_lA$$
  and this is equal to $T_l T_k^T$ since diagonal matrices commute.
  We have shown that $T_k T_l^{T}$ is symmetric.
The above equality also implies that
these matrices commute since they are simultaneously diagonalizable
(recall that $A^T=A^{-1}$).
The proof for the matrices $T_k^T T_l$ is similar, 
  with the roles of $A$ and  $B$ exchanged.
\end{proof}
\begin{remark} \label{rem:slices}
The proof for $z$-slices uses only the orthogonality of $A$ and $B$
($C$ may be an arbitrary matrix).
It relies on the fact that the slices are
simultaneously orthogonally equivalent to diagonal matrices.
\end{remark}
In the following we will not use explicitly the fact that the matrices
in Corollary~\ref{cor:slices} commute, but we will use the symmetry propery.

\subsection{Orthogonal decomposition of real tensors} \label{sec:realordi}

\begin{theorem}[simultaneous SVD] \label{SSVD}
Let  $T_1,\ldots,T_s$ be real matrices of size $n$. The following properties are equivalent:
\begin{itemize}
\item[(i)]  The matrices $T_k T_l^T$ and $T_k^T T_l$
  are symmetric for all $k,l \in \{1,\ldots,s\}$.
\item[(ii)] The $T_k$ are simultaneously orthogonally equivalent to diagonal matrices, i.e., there exist real orthogonal matrices
$U$ and $V$ such that the matrices  $U^T T_k V$ are all diagonal.
\end{itemize}
\end{theorem}
We have seen in the proof of Corollary~\ref{cor:slices} that (ii) implies (i).
As to the converse, for $k=1$ the matrix $T_1$ is an arbitrary real matrix and the required decomposition
is the singular value decomposition of $T_1$.
For $k=2$ this is exercise  2.6.P4 in~\cite{horn13}, and the general case
is Corollary~9 in~\cite{maehara11}.
This theorem implies in particular that if the matrices $T_k T_l^T$,
$T_k^T T_l$ are all symmetric then the matrices $T_k T_l^T$ must commute
(and the $T_k^T T_l$ commute as well).
\begin{proposition} \label{partialortho}
  Let $T$ be a real tensor of order 3.
  The corresponding trilinear form $f(x,y,z)$ admits a decomposition
  of the form $f(x,y,z) = g(Ax,By,Cz)$ with $A,B$ orthogonal
  and $g$ as in~(\ref{diagtensor}) iff the $z$-slices of $T$ satisfy
  the conditions of Theorem~\ref{SSVD}.(i): $T_k T_l^T$ and $T_k^T T_l$
  are symmetric for all $k,l \in \{1,\ldots,n\}$.
\end{proposition}
\begin{proof}
If $f(x,y,z) = g(Ax,By,Cz)$ with $A,B$ orthogonal
and $g$ as in~(\ref{diagtensor}),
we have already seen (see Remark~\ref{rem:slices})
that the matrices $T_k T_l^T$ and $T_k^T T_l$
are symmetric.

Conversely, if these matrices are symmetric then by Theorem~\ref{SSVD}
there are real orthogonal matrices $U,V$ such that the matrices
$U^T T_k V$ are all diagonal. Let $h(x,y,z)=f(Ux,Vy,z)$.
By Proposition~\ref{slices} the $z$-slices of $h$ are diagonal, i.e.,
we have $$\frac{\partial h}{\partial z_k} = x^TD_ky$$
where $D_k$ is a diagonal matrix.
This implies that the trilinear form  $h$ can be written as
\begin{equation} \label{ABId}
  h(x,y,z)=\sum_{i=1}^n x_i y_i c_i(z)
\end{equation}
where the $c_i$ are linear forms
(indeed, the presence of any cross-product $x_i y_jz_k$ with $i \neq j$ in $h$
would give rise to a non-diagonal entry in $D_k$).
We have shown that that $h(x,y,z)=g(x,y,Cz)$ where
$g(x,y,z)=\sum_{i=1}^n x_i y_i z_i$ and $C$ is the matrix with the $c_{ik}$
as entries. As a result we have $f(x,y,z)=h(U^Tx,V^Ty,z)=g(U^Tx,V^Ty,Cz).$
\end{proof}

The next result is the main result of Section~\ref{sec:realordi}.
In particular, we recover
the result from~\cite{boralevi17} that the set of real tensors admitting
an orthogonal decomposition can be defined by equations of degree 2.
\begin{theorem} \label{realortho}
  For a real tensor $T$ of order 3, the following properties are equivalent:
  \begin{itemize}
  \item[(i)] $T$ admits an orthogonal decomposition.
  \item[(ii)] The $x$, $y$ and $z$ slices of $T$ satisfy
  the conditions of Theorem~\ref{SSVD}.(i).
    \end{itemize}
\end{theorem}
For instance, if we denote by $Z_1,\ldots,Z_n$ the $z$-slices of $T$,
the matrices $Z_k Z_l^T$ and $Z_k^T Z_l$ must all be symmetric (and likewise
for the $x$ and $y$ slices).
Toward the proof of this theorem we need the following lemma.
\begin{lemma} \label{lem:realortho}
  Let $f(x,y,z)$ be a real trilinear form and let $A,B,C$ be orthogonal
  matrices. If the $x$, $y$ and $z$ slices of  $f$ satisfy the conditions
  of Theorem~\ref{SSVD}.(i) then the same is true of the form
  $h(x,y,z)=f(Ax,By,Cz)$.
\end{lemma}
\begin{proof} We can obtain obtain $h$ from $f$ in 3 steps: first perform
  the linear transformation on the $x$ variables, then on the $y$ variables
  and finally on the $z$ variables. It is therefore sufficient to
  prove the lemma for 
  $h(x,y,z)=f(Ax,y,z)$.
  Let $Z_1,\ldots,Z_n$  be the $z$-slices of $f$ and $Z'_1,\ldots,Z'_n$
  those of $h$. By Proposition~\ref{slices} we have $Z'_k=A^TZ_k$ so that
  $Z'_k {Z'_l}^T = A^T (Z_k Z_l^T) A$. This matrix is symmetric since $Z_k Z_l^T$
  is symmetric, and the same is true of  ${Z'_k}^T Z'_l = Z_k^TAA^TZ_l=Z_k^TZ_l$.
  The $z$-slices of $h$ therefore satisfy the hypotheses of
  Theorem~\ref{SSVD}.(i), and a similar argument applies to the $y$-slices.

  Finally, the $x$-slices of $h$ are $X'_k=\sum_{i=1}^n a_{ik}X_i$ where $X_1,\ldots,X_n$
  are the $x$-slices of $f$. As a result,
  $$X'_k{X'_l}^T = \sum_{i,j=1}^n a_{ik}a_{jl}X_iX_j^T.$$
  This matrix is symmetric since the matrices $X_i X_j^T$ are symmetric by
  hypothesis. A similar computation shows that ${X'_k}^T{X'_l}$ is symmetric
  as well. We have therefore shown that all slices of $h$ satisfy
  the hypotheses of Theorem~\ref{SSVD}.(i).
\end{proof}

\begin{proof}[Proof of Theorem~\ref{realortho}]
  If $T$ admits an orthogonal decomposition, Proposition~\ref{partialortho}
  shows that the $x$, $y$ and $z$ slices satisfy
  the conditions of Theorem~\ref{SSVD}.(i).
 For the converse 
 we begin with a special case:  let us assume that the trilinear form $f$
 associated to $T$ has the same form
  as $h$ in~(\ref{ABId}), i.e.,
  \begin{equation} \label{fABId}
  f(x,y,z)=\sum_{i=1}^n x_i y_i c_i(z)
  \end{equation}
  where the $c_i$ are linear forms.
  Differentiating this expression shows that
  $$\frac{\partial f}{\partial x_k} = y_k c_k(z)$$
  and the $x$-slices of $f$ are therefore the rank-one matrices $X_k = D_k C$
  where $C$ is the matrix with the $c_{ij}$ as entries, and $D_k$ is the diagonal matrix with an entry equal to 1 at row $k$ and column $k$,
  and zeroes elsewhere. By hypothesis the matrix $X_kX_l^T=D_k C C^T D_l$
  must be symmetric. But this is a matrix with at most one nonzero entry,
  located at row $k$ and column $l$. We conclude that $D_k C C^T D_l=0$
  for $k \neq l$, i.e., any two distinct rows of $C$ are orthogonal.
  Normalizing the rows of $C$, we can write
  $$f(x,y,z)=\sum_{i=1}^n \alpha_i x_i y_i c'_i(z)$$
  where $C'$ is an orthogonal matrix.
  Note that the rows of $C$ that are identically zero require a special
  treatment since they cannot be normalized.
  If there is one such row we replace it in $C'$ by a unit vector
  that is orthogonal to the $n-1$ other rows of $C$ (there are 2 choices)
  and the corresponding coefficient $\alpha_i$ is set to 0.
  More generally, if there are several null rows we pick an orthonormal
  basis of the orthogonal complement of the span of the rows of $C$,
  and all the corresponding coefficient $\alpha_i$ are set to 0.
  
  We have therefore shown that $f(x,y,z)=g(x,y,C'z)$
  where $g$ is as in~(\ref{diagtensor}), i.e., we have obtained an orthogonal
  decomposition of $f$.

  The last step of the proof is  a reduction from the general case
  to~(\ref{fABId}).
  Let $T$ be any tensor satisfying property (ii)
  in the statement of the Theorem,
  and let $f$ be the corresponding trilinear form.
  We have seen in the proof of Proposition~\ref{partialortho}
  that there are two orthogonal matrices $U,V$ such that the form
  $h(x,y,z)=f(Ux,Vy,z)$ is 
  as in~(\ref{ABId}).
  By Lemma~\ref{lem:realortho}, $h$ also satisfies the hypotheses of
  Theorem~\ref{realortho} and we have therefore shown in the previous step
  of the proof that $h$ admits an orthogonal decomposition.
  The same is true of $f$ since $f(x,y,z)=h(U^Tx,V^Ty,z)$.
\end{proof}

\subsection{Orthogonal decomposition of complex tensors} \label{sec:cordi}

The following result of Choudhury and Horn~\cite{chou87} is an analogue
of Theorem~\ref{SSVD} for
complex orthogonal equivalence.
We state it for square matrices because that is sufficient for our purposes,
but it can be generalized to rectangular matrices
(see~\cite{chou87} for details).
\begin{theorem}[simultaneous diagonalization by complex orthogonal equivalence]
  \label{SDCOE} Let  $T_1,\ldots,T_s$ be complex matrices of size $n$.
  The $T_k$ are simultaneously orthogonally equivalent to diagonal matrices
  (i.e., there exist complex orthogonal matrices
  $U$ and $V$ such that the matrices  $U^T T_k V$ are all diagonal) if and only
  if the following conditions are satisfied:
  \begin{itemize}
  \item[(i)] For each $k$, $T_k^T T_k$ is diagonalizable and
    $\rk T_k = \rk T_k^T T_k$.
  \item[(ii)] The matrices $T_k T_l^T$ and $T_k^T T_l$
  are symmetric for all $k,l \in \{1,\ldots,s\}$.
  \end{itemize}
\end{theorem}
Condition (i) is the necessary and sufficient condition for each $T_k$ to be
(individually) orthogonally equivalent to a diagonal matrix
(\cite{chou87}, Theorem 2).
It does not appear in Theorem~\ref{SSVD} because this condition is automatically
satisfied by real matrices.
When condition~(i) holds, it is shown in Theorem 9 of~\cite{chou87}
that condition (ii) is necessary and sufficient for the $T_k$
to be simultaneously orthogonally equivalent to diagonal matrices.

We continue with an analogue of Proposition~\ref{partialortho}.
\begin{proposition} \label{partialorthoC}
  Let $T$ be a complex tensor of order 3.
  The corresponding trilinear form $f(x,y,z)$ admits a decomposition
  of the form $f(x,y,z) = g(Ax,By,Cz)$ with $A,B$ orthogonal
  and $g$ as in~(\ref{diagtensor}) iff the $z$-slices of $T$ satisfy
  conditions (i) and (ii) of Theorem~\ref{SDCOE}.
\end{proposition}
\begin{proof}
 Assume first that $f(x,y,z) = g(Ax,By,Cz)$ with $A,B$ orthogonal
 and $g$ as in~(\ref{diagtensor}).
 By Proposition~\ref{slices} the $z$-slices $T_1,\ldots,T_n$ are of the
 form $T_k = A^TD_kB$ where $D_k$ is diagonal, i.e., the $T_k$ are simultaneously orthogonally equivalent to diagonal matrices. They must therefore satisfy
 the conditions of Theorem~\ref{SDCOE}.

 Conversely, assume that the conditions of this theorem are satisfied.
 Then there are complex orthogonal matrices $U,V$ such that the matrices
 $U^T T_k V$ are all diagonal. We can consider the polynomial
 $h(x,y,z)=f(Ux,Vy,z)$ and conclude exactly
 as in the proof of Proposition~\ref{partialortho}.
\end{proof}
We proceed to the main result of Section~\ref{sec:cordi}.
This is an analogue of Theorem~\ref{realortho}.
\begin{theorem} \label{cortho}
  A complex tensor $T$ of order 3 admits an orthogonal decomposition iff the $x$, $y$ and $z$ slices of $T$ all satisfy conditions (i) and (ii)
  of Theorem~\ref{SDCOE}, namely, the $x$-slices must satisfy the conditions:
  \begin{itemize}
  \item[(i)] for each $k$, $X_k^T X_k$ is diagonalizable and
    $\rk X_k = \rk X_k^T X_k$,
  \item[(ii)] the matrices $X_k X_l^T$ and $X_k^T X_l$
  are symmetric for all $k,l \in \{1,\ldots,n\}$,
  \end{itemize}
  and likewise for the $y$ and $z$ slices.
  \end{theorem}
For the proof we naturally need an analogue of Lemma~\ref{lem:realortho}:
\begin{lemma} \label{lem:cortho}
  Let $f(x,y,z)$ be a complex trilinear form and let $A,B,C$ be orthogonal
  matrices. If the $x$, $y$ and $z$ slices of  $f$ satisfy conditions (i)
  and (ii) of Theorem~\ref{SDCOE} then the same is true of the form
  $h(x,y,z)=f(Ax,By,Cz)$.
\end{lemma}
\begin{proof}
  Since the slices of $f$ satisfy condition (ii) of Theorem~\ref{SDCOE},
  the same must be true of $h$. Indeed, the proof of Lemma~\ref{lem:realortho}
  applies verbatim to the present situation.
  It therefore remains to deal with condition (i).
  Like in Lemma~\ref{lem:realortho} it suffices to consider the case
  $h(x,y,z)=f(Ax,y,z)$.
  Let us denote again by $Z_1,\ldots,Z_n$  the $z$-slices of $f$ and by
  $Z'_1,\ldots,Z'_n$ the $z$-slices  of $h$.
  We saw that $Z'_k=A^TZ_k$ by Proposition~\ref{slices} so that
  $Z'_k {Z'_k}^T = A^T (Z_k Z_k^T) A= A^{-1} (Z_k Z_k^T)A$.
  According to condition (i) $Z_k Z_k^T$ is diagonalizable,
  so the same is true of $Z'_k {Z'_k}^T$.
  Moreover $\rk Z'_k {Z'_k}^T = \rk Z_k Z_k^T = \rk Z_k = \rk Z'_k$ 
  and this completes the proof that the $Z'_k$ satisfy condition (i).
  A similar argument applies to the $y$-slices.

  For the $x$-slices we will use the fact that condition (i)
  of Theorem~\ref{cortho} is the condition for each slice
  to be (individually) orthogonally equivalent to a diagonal matrix.\footnote{We could also have used this argument for the $y$ and $z$ slices.}
  As pointed out before, this is shown in \cite[Theorem~2]{chou87}.
  As shown in the proof of Lemma~\ref{lem:realortho}, each $x$-slice
  $X'_k$ of $h$ is a linear combination of the $x$-slices $X_1,\ldots,X_n$
  of $f$. These slices satisfy condition (i) and (ii) and are therefore
  simultaneously orthogonally equivalent to diagonal matrices
  by Theorem~\ref{SDCOE}.
  Any linear combination of these matrices, and in particular $X'_k$,
  is therefore orthogonally equivalent to a diagonal matrix.
  The $X'_k$ must therefore satisfy condition (i) by~\cite[Theorem~2]{chou87}
  (note that we only use the ``easy'' direction of this theorem here).
\end{proof}

\begin{proof}[Proof of Theorem~\ref{cortho}]
 If $T$ admits an orthogonal decomposition, Proposition~\ref{partialorthoC}
  shows that the $x$, $y$ and $z$ slices satisfy
  conditions (i) and (ii) of Theorem~\ref{SDCOE}.

  For the converse, we begin as in the proof Theorem~\ref{realortho}
  with the case where the trilinear form $f$ associated to $T$
  is as in~(\ref{fABId}). This case can be treated in the same way except
  for one important difference. In the proof of Theorem~\ref{realortho}
  we pointed out that the null rows of $C$ must be treated separately
  because they cannot be normalized. Over $\cc$ there is a further
  complication: there might be rows $c_k$ such that $c_k^Tc_k=0$ but
  $c_k \neq 0$; such rows could not be normalized.
  Fortunately, it turns out that there are no such rows in $C$.
  Recall indeed from the  proof of Theorem~\ref{realortho}
  that the $k$-th $x$-slice of $T$ is $X_k=D_kC$,
  where $D_k$ is the diagonal matrix with an entry equal to 1
  at row $k$ and column $k$, and zeroes elsewhere.
  In other words, row $k$ of $X_k$ is equal to $c_k$ and all other rows
 are identically 0. Moreover $X_k^T X_k$ has one entry (at row $k$ and
  column $k$) equal to $c_k^Tc_k$ and only 0's elsewhere.
  By condition (i) of Theorem~\ref{SDCOE} we must
  have $\rk X_k = \rk X_k^T X_k$.
  It follows that $c_k^Tc_k=0$ implies $c_k=0$.
  We have therefore shown that all rows of $C$ can be normalized
  except the null rows. Morever, any two distinct rows are
  orthogonal as in the proof of Theorem~\ref{realortho}.
  We can therefore conclude essentially as in that proof: the set of normalized non-null rows of $C$
  is an orthonormal family, and can therefore be extended to an orthonormal
  basis of $\cc^n$ by Witt's theorem (Theorem~\ref{witt}
  and Corollary~\ref{wittcor}).
    Again, the coefficients $\alpha_i$ corresponding to the new vectors
  in this basis are set to 0.

  It remains to reduce the general case of Theorem~\ref{cortho}
  to~(\ref{fABId}) and this can be done essentially as in the proof of
  Theorem~\ref{realortho}. Indeed, let $Z_1,\ldots,Z_n$
  be the $z$-slices of $T$. By Theorem~\ref{SDCOE} there are
  orthogonal matrices $U,V$ such that all the matrices $D_k = U^T Z_k V$
  are diagonal. We set $h(x,y,z)=f(Ux,Vy,z)$ like in the proof
  of Proposition~\ref{partialortho} and the $z$-slices of $h$
  are the diagonal matrices $D_k$ by Proposition~\ref{slices}.
  This implies that $h$ is as in~(\ref{ABId}).
  Moreover, $h$ satisfies conditions (i) and (ii) of Theorem~\ref{SDCOE}
  by Lemma~\ref{lem:cortho}. This completes the reduction, and the proof
  of Theorem~\ref{cortho}.
\end{proof}

\subsection{Closure properties} \label{closure}

In Section~\ref{sec:cordi} we gave a characterization of the set
$\OT_n(\cc)$
of tensors $T \in \cc^{n \times n \times n}$ that admit an orthogonal decomposition.
In this section we show that $\OT_n(\cc)$ is not closed, and we find
a a nontrivial family of tensors in the closure (Theorem~\ref{geninclosure}).
Then we find some of the equations that are satisfied by tensors in 
$\overline{\OT_n(\cc)}$ (Theorem~\ref{otbar}) and we show that these
equations do not characterize the closure completely (Theorem~\ref{outclosureot}).

First we show that $\OT_2(\cc)$ is not closed by exhibiting a tensor
belonging to $\overline{\OT_2(\cc)}$ but not to $\OT_2(\cc)$.
This is in fact the same tensor as in Example~\ref{borderex}
but we view it as an ordinary tensor instead of a symmetric tensor.
\begin{example} \label{borderexordi}
  Let $f^1(x_1,x_2,y_1,y_2,z_1,z_2)=(x_1+ix_2)(y_1+iy_2)(z_1+iz_2)$.
  The $z$-slices of $f^1$ are: 
$$Z_1 = \begin{pmatrix}
  1 & i\\
  i & -1
\end{pmatrix},
Z_2 = i\begin{pmatrix}
  1 & i\\
  i & -1
\end{pmatrix}.$$
and the $x$ and $y$ slices are of course the same.
We can apply Theorem~\ref{cortho} to show that $f^1 {\not \in} \OT_2(\cc)$.
Indeed, the second part of condition (i) of Theorem~\ref{SDCOE} is violated:
$\rk Z_1 =1$ but $\rk Z_1^T Z_1 = \rk Z_1^2 = 0$.

In order to show that $f^1 \in \overline{\OT_2(\cc)}$, consider the polynomials
$$g_{\epsilon}=x_1y_1z_1+\epsilon x_2y_2z_2$$
and $f^1_{\epsilon}=g_{\epsilon}(A_{\epsilon}x,A_{\epsilon}y,A_{\epsilon}z)$
where $$A_{\epsilon}=\begin{pmatrix}
  1 & i+\epsilon\\
  1+\epsilon^2 & i-\epsilon
\end{pmatrix}$$
is the same matrix as in Example~\ref{borderex}.
Since $f^1 = \lim_{\epsilon \rightarrow 0} f^1_{\epsilon}$ it remains to show that
$f^1_{\epsilon} \in \OT_2(\cc)$ for all $\epsilon$ sufficiently close to 0.
We have seen in Example~\ref{borderex} that $A_{\epsilon} = D_{\epsilon}U_{\epsilon}$ where $D_{\epsilon}$
is a diagonal matrix and $U_{\epsilon}$ orthogonal.
Hence we have
\begin{equation} \label{epsilondec}
  f^1_{\epsilon}=g_{\epsilon}(D_{\epsilon}U_{\epsilon}x,D_{\epsilon}U_{\epsilon}y,
  D_{\epsilon}U_{\epsilon}z)=h_{\epsilon}(U_{\epsilon}x,U_{\epsilon}y,U_{\epsilon}z)
  \end{equation}
where $h_{\epsilon}=\alpha_{\epsilon}x_1y_1z_1+\beta_{\epsilon}x_2y_2z_2$
for some appropriate coefficients $\alpha_{\epsilon}, \beta_{\epsilon}$.
We conclude that~(\ref{epsilondec}) provides as needed
an orthogonal decomposition of $f^1_{\epsilon}$.
\end{example}
We can also give less symmetric examples of polynomials on the boundary
of $\OT_2(\cc)$.
\begin{example} \label{borderex2}
 Let $f^2(x_1,x_2,y_1,y_2,z_1,z_2)=(x_1+ix_2)(y_1+iy_2)z_1.$
 This polynomial has the same first z-slice as the polynomial $f$ of Example~\ref{borderexordi}; this shows that $f^2 {\not \in} \OT_2(\cc)$.
 In order to show that $f^2 {\in} \overline{\OT_2(\cc)}$ we can proceed
 as in the previous example.
 Indeed we have $f^2 = \lim_{\epsilon \rightarrow 0} f^2_{\epsilon}$
 where $f^2_{\epsilon}=g_{\epsilon}(A_{\epsilon}x,A_{\epsilon}y,z) =
 g_{\epsilon}(D_{\epsilon}U_{\epsilon}x,A_{\epsilon}U_{\epsilon}y,z)$.
 From this representation of $f^2_{\epsilon}$
 we obtain an orthogonal decomposition in the same way as before.
\end{example}
\begin{example} \label{borderex3}
  Another similar example is: $$f^3(x_1,x_2,y_1,y_2,z_1,z_2)=(x_1+ix_2)y_1z_1.$$
  The $z$-slices of $f^3$ are: 
$$Z_1 = \begin{pmatrix}
  1 & 0\\
  i & 0
\end{pmatrix},
  Z_2 = 0.$$
  We have $Z_1^TZ_1=0$ and we conclude that our polynomial does not belong to
  $\OT_2(\cc)$ for the same reason as in Example~\ref{borderexordi}.
  In order to show that $f^3 \in \overline{\OT_2(\cc)}$, consider the polynomial
  $f^3_{\epsilon}=g_{\epsilon}(A_{\epsilon}x,y,z)$
  where $g_{\epsilon}$ and $A_{\epsilon}$ are as in the two previous examples.
  Since $f^3 = \lim_{\epsilon \rightarrow 0} f^3_{\epsilon}$ it remains to show that
$f^3_{\epsilon} \in \OT_2(\cc)$ for all $\epsilon$ sufficiently close to 0.

  We have seen 
  that $A_{\epsilon} = D_{\epsilon}U_{\epsilon}$ where $D_{\epsilon}$
is a diagonal matrix and $U_{\epsilon}$ orthogonal.
Hence we have
\begin{equation} \label{epsilondec2}
  f^3_{\epsilon}=g_{\epsilon}(D_{\epsilon}U_{\epsilon}x,y,z)
  =g'_{\epsilon}(U_{\epsilon}x,y,z)
  \end{equation}
where $h'_{\epsilon}=\alpha'_{\epsilon}x_1y_1z_1+\beta'_{\epsilon}x_2y_2z_2$
for some appropriate coefficients $\alpha'_{\epsilon}, \beta'_{\epsilon}$.
We conclude that~(\ref{epsilondec2}) provides as needed
an orthogonal decomposition of $f^3_{\epsilon}$.
\end{example}

Like in Section~\ref{sclosure} we can build on these examples to exhibit
more elements of $\overline{\OT_n(\cc)}$.
\begin{definition} \label{def:odeco}
We denote by $\ODECO_n(\cc)$ the set of tensors of order 3 that
    can be written as $\sum_{j=1}^k  u_j \otimes v_j \otimes w_j$
    where  each of the the 3 lists $(u_1,\ldots,u_k)$, $(v_1,\ldots,v_k)$,
    $(w_1,\ldots,w_k)$ is made of $k$ linearly independent,
    pairwise orthogonal vectors
    in $\cc^n$.
    We use the same notation for the corresponding set of trilinear forms
    in $\cc[x_1,\ldots,x_n,y_1,\ldots,y_n,z_1,\ldots,z_n]$.
\end{definition}

The next proposition gives a characterization of $OT_n(\cc)$ in the style
of Definition~\ref{def:odeco}. Compare with Proposition~\ref{waringtensor},
where for the real field we did have to introduce explicitly a non-isotropy
requirement for the vectors $u_i$, $v_i$, $w_i$ appearing in the decomposition.
\begin{proposition} \label{nonisotropic}
  $OT_n(\cc)$ is equal to the set of order 3 tensors which admit a decomposition
  of the form $\sum_{j=1}^k  u_j \otimes v_j \otimes w_j$ for some $k \leq n$,
    where  each of the 3 lists $(u_1,\ldots,u_k)$, $(v_1,\ldots,v_k)$,
    $(w_1,\ldots,w_k)$ is made of $k$ 
    pairwise orthogonal non-isotropic vectors of $\cc^n$.
\end{proposition}
We skip the proof of this proposition because it is entirely parallel
to the proof of Proposition~\ref{snonisotropic} for symmetric tensors.
\begin{theorem} \label{geninclosure}
  For every $n \geq 2$ we have $\OT_n(\cc) \subseteq \ODECO_n(\cc) \subseteq \overline{\OT_n(\cc)}$.
  The first inclusion is strict for every $n \geq 2$.
\end{theorem}
\begin{proof}
  The first inclusion follows from Definition~\ref{def:odeco} and Proposition~\ref{nonisotropic}.
  Any one of the 3 examples at the beginning of Section~\ref{sclosure}
  shows that the inclusion is strict for $n=2$.
  Like in the proof of Theorem~\ref{inclosure} this can be extended to
  any $n >2$ by adding dummy variables.
  
  For the second inclusion,
  consider a tensor $T \in \ODECO_n(\cc)$ and the corresponding decomposition
  $T = \sum_{j=1}^k  u_j \otimes v_j \otimes w_j$. We can approximate
  to an arbitrary precision the tuple $(u_1,\ldots,u_k)$ by tuples
  $(u'_1,\ldots,u'_k)$ satisfying the property of Lemma~\ref{approx}.
  We can also approximate the tuples $(v_1,\ldots,v_k)$ and
  $(w_1,\ldots,w_k)$ by tuples $(v'_1,\ldots,v'_k)$ and
  $(w'_1,\ldots,w'_k)$ satisfying the same property.
  In this way we approximate $T$ to an arbitrary precision by the tensors
  $\sum_{j=1}^k  u'_j \otimes v'_j \otimes w'_j$,
  and these tensors belong to $OT_n(\cc)$ by Proposition~\ref{nonisotropic}.
 \end{proof}
One could show that
the second inclusion is strict for every large enough~$n$
by adapting the proof of a result about symmetric tensors
in an earlier version of this paper:
Theorem~30 in~\cite{Koi19v2}.
In that theorem the strict inclusion
$\SODECO_n(\cc) \subsetneq \overline{\OW_n(\cc)}$
was obtained for large enough $n$
for a polynomial of ``high Waring rank''. 
For Theorem~\ref{geninclosure} one would start instead from
a multilinear polynomial of ``high tensor rank'' (the details are omitted).
It would be interesting to find out whether the second inclusion of
Theorem~\ref{geninclosure} is strict for {\em all} $n \geq 2$.
Recall that the corresponding result for symmetric tensors was established
in Theorem~\ref{inclosure} of the present paper.

In order to complete the parallel with the study of symmetric tensors
in Section~\ref{sec:waring} it remains to investigate the closure properties
of $\OT_n(\cc)$ ``from above''.
Like in Section~\ref{sec:asd} this will be done thanks to a connection
with the ASD property.
\begin{proposition} \label{otasd}
  Let $X_1,\ldots,X_n$ be the $x$-slices of
  a tensor in $\overline{\OT_n(\cc)}$.
  The $n^2$ matrices $X_k^TX_l$ ($1 \leq k,l \leq n$) are symmetric and
  approximately simultaneously diagonalizable  (ASD).
  Likewise, the $n^2$ matrices $X_kX_l^T$ are symmetric and ASD.
  The same properties also hold for the $y$ and $z$-slices.
\end{proposition}
\begin{proof}
  Let $X_1,\ldots,X_n$ be the $x$-slices of a tensor $T \in \OT_n(\cc)$.
  By Theorem~\ref{SDCOE} and Theorem~\ref{cortho}
  there are orthogonal matrices $U$ and $V$ such that the matrices
  $D_k = U^TX_k V$ are all diagonal (this is actually the easier direction
  of Theorem~\ref{cortho}).
    Therefore $X_k^T X_l = VD_kD_l V^T$. In particular, these $n^2$ matrices
  are symmetric and simultaneously diagonalizable. Passing to the limit
  shows that for a tensor in $\overline{\OT_n(\cc)}$, the corresponding
  $n^2$ matrices must be symmetric and ASD.
  The same argument applies to $X_k X_l^T= UD_k D_l U^T$ and to the $y$ and $z$ slices.
\end{proof}
\begin{remark} \label{rem:otasd}
  This section deals with ordinary (possibly non symmetric) tensors,
  but still Proposition~\ref{otasd} shows that
  we need to use the ASD property for symmetric matrices only.
  Moreover we only need to consider approximations by simultaneously
  diagonalizable {\em symmetric} matrices since the matrices $VD_kD_l V^T$
  in the proof of Proposition~\ref{otasd} are symmetric.
 \end{remark}
By  Proposition~\ref{otasd} and Theorem~\ref{ASDimplies} we have the following analogue of Theorem~\ref{owbar}:
\begin{theorem} \label{otbar}
  The $x$-slices $X_1,\ldots,X_n$ of a tensor in $\overline{\OT_n(\cc)}$
  must satisfy the following properties:
  \begin{itemize}
 \item[(i)]  The matrices $X_k^TX_l$ are symmetric and pairwise commute.
 \item[(ii)] The subalgebra of $M_n(\cc)$ generated by these $n^2$ matrices
   and by the identity matrix is of dimension at most $n$.
 \item[(iii)] The centralizer of these matrices is of dimension at least $n$.
  \end{itemize}
 The same properties are satisfied by the matrices $X_k X_l^T$ and by the $y$ and $z$ slices.
  \end{theorem}
Finally, we show that the converse of this theorem does not hold.
We will need the following lemma, which follows from Definition~\ref{orthodef}
and the fact that the complex orthogonal group is of dimension $n(n-1)/2$.
\begin{lemma} \label{dimortho}
 $\dim \OT_n(\cc) \leq 3n(n-1)/2+n$.
\end{lemma}
\begin{theorem} \label{outclosureot}
  For every large enough $n$ there is a tensor of order 3 and size $n$
  which satisfies all the properties of Theorem~\ref{otbar}
  but does not belong to $\overline{\OT_n(\cc)}$.
\end{theorem}
\begin{proof}
  As a counterexample we will construct a symmetric tensor $S$,
  which we are of course free to view as an ordinary tensor.
  In fact, like in the proof of Theorem~\ref{outclosure}
  we will construct a tensor of the form
  \begin{equation} \label{counterot}
    S=\sum_{i=1}^r u_i^{\otimes 3}
   \end{equation}
  where the $u_i$ belong to a totally isotropic
  subspace $V \subseteq \cc^n$ with $\dim V = \lfloor n/2 \rfloor$.
  Let us fix such a $V$. Since $r$ is arbitrary, the set ${\cal S}$
  of tensors of form~(\ref{counterot}) is a linear space. Its  dimension
  is equal to ${v+2 \choose 3}$ where $v = \lfloor n/2 \rfloor$
  (this is the dimension of the space of homogeneous
  polynomials of degree 3 in $v$ 
  variables).

  The $x$, $y$ and $z$ slices of any $S \in \cal S$ are the same since $S$ is symmetric.
  Let us denote them by $S_1,\ldots,S_n$.
  We claim that the matrices $S_k S_l^T$ and $S_k^T S_l$
  considered in Theorem~\ref{otbar} are all equal to 0.
  They will therefore trivially satisfy properties (i), (ii) and (iii).
  The proof of the claim is simple: by symmetry of $S$ we have $S_k^T=S_k$
  and $S_l^T = S_l$, but by Proposition~\ref{zeroproduct}
  the product of any two
  slices of $S$ is equal to 0.
  Therefore it remains to find an $S \in \cal S$
  which is not in $\overline{\OT_n(\cc)}$.
  Such an $S$ is guaranteed to exist as soon as
  $\dim \overline{\OT_n(\cc)} = \dim \OT_n(\cc) < \dim {\cal S}$.
  We have shown that $\dim {\cal S}$ is of order $n^3/48$, but $\dim \OT_n(\cc)$
  is quadratically bounded by Lemma~\ref{dimortho}. Hence we will have
  $\dim \OT_n(\cc) < \dim {\cal S}$ for every large enough $n$
  (one can check that it suffices to take $n \geq 68$).
\end{proof}
In this paper, Theorem~\ref{outclosureot} is the only result for ordinary
tensors with a (slightly) simpler proof than its counterpart
for symmetric tensors
(Theorem~\ref{outclosure}).
This is due to the fact that we could design the counterexample $S$ in the
proof so that $S_k S_l^T = S_k^T S_l =0$.

\small

\section*{Acknowledgements}

Nicolas Ressayre made some useful comments on an early version of this paper.
I would also like to thank  Kevin O'Meara for the encouragements, 
and Roger Horn for sharing his proof of Theorem \ref{th:simdiag} and its generalizations.
The referee's comments led to a strengthening of Theorem~\ref{inclosure} 
and several improvements in the presentation of the paper.

\bibliographystyle{plain}

\end{document}